\documentclass[11pt,leqno]{amsart}
\usepackage{a4wide}
\usepackage{amssymb,upgreek}
\usepackage{amsmath}
\usepackage{amsthm}
\usepackage{mathrsfs}
\usepackage{mathabx}
\usepackage{enumitem}
\usepackage{bm,euscript,csquotes,comment}
\usepackage[curve]{xypic}
\usepackage[usenames,dvipsnames]{xcolor}

\usepackage{hyperref}

\hypersetup{colorlinks,  citecolor=ForestGreen,  linkcolor=Mahogany, urlcolor=black}

\newcommand{\id}{\operatorname{id}}
\newcommand{\im}{\operatorname{im}}

\newcommand{\ob}{\operatorname{ob}}

\newcommand{\pr}{\operatorname{pr}}
\newcommand{\ev}{\operatorname{ev}}

\newcommand{\adj}{\operatorname{adj}}

\newcommand{\res}{\operatorname{res}}

\newcommand{\Hom}{\operatorname{Hom}}
\newcommand{\End}{\operatorname{End}}
\newcommand{\Ext}{\operatorname{Ext}}

\newcommand{\Mod}{\operatorname{Mod}}

\newcommand{\ind}{\operatorname{ind}}

\newcommand{\QCoh}{\operatorname{QCoh}}

\makeatletter
\def\cal@symb#1|#2{\expandafter\def\csname #2#1\endcsname{\mathcal{#1}}}
\def\calsymbols#1#2{\@for\@tmpz:=#2\do{\expandafter\cal@symb\@tmpz|{#1}}}

\def\b@symb#1|#2{\expandafter\def\csname #2#1\endcsname{\mathbf{#1}}}
\def\bsymbols#1#2{\@for\@tmpz:=#2\do{\expandafter\b@symb\@tmpz|{#1}}}

\def\bb@symb#1|#2{\expandafter\def\csname #2#1\endcsname{\mathbb{#1}}}
\def\bbsymbols#1#2{\@for\@tmpz:=#2\do{\expandafter\b@symb\@tmpz|{#1}}}

\def\frak@symb#1|#2{\expandafter\def\csname #2#1\endcsname{\mathfrak{#1}}}
\def\fraksymbols#1#2{\@for\@tmpz:=#2\do{\expandafter\frak@symb\@tmpz|{#1}}}

\makeatother
\calsymbols{c}{A,B,C,D,E,F,G,H,I,J,K,L,M,N,O,P,Q,R,S,T,U,V,W,X,Y,Z}
\bsymbols{bf}{A,B,C,D,E,F,G,H,I,J,K,L,M,N,O,P,Q,R,S,T,U,V,W,X,Y,Z}
\fraksymbols{f}{p,a,b,c,d,e,f,g,h,j,k,l,o,p,q,r,s,t,u,v,w,x,y,z,A,B,C,L}
\bbsymbols{b}{A,B,C,D,E,F,G,H,I,J,K,L,M,N,O,P,Q,R,S,T,U,V,W,X,Y,Z}

\theoremstyle{plain}

\newtheorem{theorem}{Theorem}[section]
\newtheorem{corollary}[theorem]{Corollary}
\newtheorem{lemma}[theorem]{Lemma}
\newtheorem{proposition}[theorem]{Proposition}

\newtheorem{definition}[theorem]{Definition}

\theoremstyle{remark}
\newtheorem{remark}[theorem]{Remark}

\title{\textbf{The central sheaf of a Grothendieck category}}
\author{Konstantin Ardakov, Peter Schneider}
\date{\today}

\address{Mathematical Institute, University of Oxford, Oxford OX2 6GG}
\email{ardakov@maths.ox.ac.uk}
\urladdr{http://people.maths.ox.ac.uk/ardakov/}

\address{ Universit\"at M\"unster,  Mathematisches Institut,  Einsteinstr. 62, 48291 M\"unster, Germany}
\email{pschnei@wwu.de}
\urladdr{https://www.uni-muenster.de/Arithm/schneider/}

\begin{document}
\begin{abstract} The center $Z(\cA)$ of an abelian category $\cA$ is the endomorphism ring of the identity functor on that category. A localizing subcategory of a Grothendieck category $\cC$ is said to be \emph{stable} if it is stable under essential extensions. The set $\mathbf{L}^{st}(\cC)$ of stable localizing subcategories of $\cC$ is partially ordered under reverse inclusion. We show $\cL \mapsto Z(\cC/\cL)$ defines a sheaf of commutative rings on $\mathbf{L}^{st}(\cC)$ with respect to finite coverings. When $\cC$ is assumed to be locally noetherian, we also show that the sheaf condition holds for arbitrary coverings.
\end{abstract}

\maketitle

\section{Introduction}

\subsection{Background}\label{sec:bgnd}
In his thesis \cite{Gab}, Gabriel proved that a noetherian scheme $X$ can be reconstructed from the category of quasi-coherent sheaves $\QCoh(X)$ on $X$. This category is an example of a \emph{Grothendieck category}: it is abelian, it has small coproducts (hence small colimits) such that filtered colimits are exact, and it has a generator. It also enjoys two additional properties: it is locally noetherian, and every localizing subcategory is stable.

To prove his reconstruction theorem, Gabriel introduces the injective spectrum $\mathbf{Sp}(\cC)$ of any Grothendieck category $\cC$ (see $\S \ref{sec:StabMain}$ below for more details) and shows that for $\cC = \QCoh(X)$, this spectrum is in a natural bijection with $X$. The open subsets $U$ of the scheme $X$, through their images in $\mathbf{Sp}(\QCoh(X))$, give rise to certain localizing subcategories $\cL_U$ of $\QCoh(X)$. Then he shows that the ring of sections $\mathcal{O}_X(U)$ of the structure sheaf of the scheme $X$ over $U$ is naturally isomorphic to the center $Z(\QCoh(X)/\cL_U)$ of the quotient category $\QCoh(X)/\cL_U$. Gabriel's theorem was later generalised to arbitrary quasi-separated schemes (\cite{Ros}, \cite{Bra}).

In this paper, motivated by considerations arising from the mod-$p$ representation theory of $p$-adic reductive groups, we generalise Gabriel's approach to the construction of a `central' sheaf for any Grothendieck category $\cC$. There exist fairly natural non-trivial examples of Grothendieck categories $\cC$ where $\mathbf{Sp}(\cC)$ is the empty set, the so-called \emph{continuous spectral categories}, (cf. \cite{Pop} p. 324). Therefore it seems appropriate to take the viewpoint of `pointless topology' and try to work with the presheaf of centers $\cL \mapsto Z(\cC/\cL)$ directly.

\subsection{Stability and main results}\label{sec:StabMain}

Taking cue again from our primary example of interest, namely the category $\Mod(G)$ of smooth $k$-linear representations of a connected reductive $p$-adic group $G$ with coefficients in a field $k$, we note that although $\Mod(G)$ definitely exhibits some serious non-commutative behaviour, this category appears to be nevertheless closer to being `commutative' than, say, the category of modules over a complicated non-commutative ring such as a Weyl algebra. To support this intuition, note that at least when $k = \mathbb{C}$ is the field of complex numbers, Bernstein's famous results about the structure of $\Mod(G)$ (cf.\ \cite{Ber}) imply that the connected components of this category are Morita equivalent to algebras that are finitely generated modules over their noetherian centers.

For a general non-commutative ring $A$ the localizing subcategories (or the hereditary torsion theories, \cite{Go2}, \cite{Ste}) of the Grothendieck category $\Mod(A)$ can exhibit a rather wild behaviour. The basic outcome of this paper is the observation that, at least as far as the behaviour of centers is concerned, the situation improves considerably if we restrict our attention only to the stable localizing subcategories. We view the set $\mathbf{L}^{st}(\cC)$ of stable localizing subcategories of $\cC$ as a category with objects being the elements of $\mathbf{L}^{st}(\cC)$ and where the morphisms $\cL_2 \rightarrow \cL_1$ are given by the inclusions $\cL_1 \subseteq \cL_2$. For $\cL$ and $\cL_1, \ldots, \cL_n$ in $\mathbf{L}^{st}(\cC)$ we call $\{\cL_i\}_{1 \leq i \leq n}$ a \emph{covering} of $\cL$ if $\cL = \bigcap_i \cL_i$. Our first main result reads as follows.

\begin{theorem}\label{thm:intro-central-sheaf} For any Grothendieck category $\cC$,  $Z_\cC(\cL) := Z(\cC/\cL)$ is a sheaf on $\mathbf{L}^{st}(\cC)$.
\end{theorem}

This result will be a consequence of the following purely categorical statement, which we prove in Prop.\ \ref{prop:R-equiv}. We first recall that Gabriel (cf. \cite{Gab} p. 439) introduced for any diagram of exact functors between abelian categories
\begin{equation*}
  \xymatrix@R=0.5cm{
  \cA_0  \ar[dr]             \\
                & \cA          \\
  \cA_1  \ar[ur]               }
\end{equation*}
the category $\cA_0\, \underset{\cA}{\prod}\, A_1$ which he calls the \emph{recollement of $\cA_0$ and $\cA_1$ along $\cA$}. We apply this with $\cA_i := \cC/\cL_i$ and $\cA := \cC/(\cL_0 \vee \cL_1)$ and the corresponding quotient functors, where $\cL_0$ and $\cL_1$ are two localizing subcategories of our Grothendieck category $\cC$ and $\cL_0 \vee \cL_1$ denotes the smallest localizing subcategory which contains $\cL_0$ and $\cL_1$. It will be easy to see that the corresponding recollement receives a natural functor
\begin{equation*}
  \cC/(\cL_0 \cap \cL_1) \longrightarrow (\cC/\cL_0) \underset{\cC / (\cL_0 \vee \cL_1)}{\prod}{} (\cC/\cL_1) \ .
\end{equation*}
Our main result is that this functor is an equivalence provided $\cL_0$ and $\cL_1$ are \textit{stable}.

The restriction to the case of finite coverings in Thm.\ \ref{thm:intro-central-sheaf} can be removed if we impose a finiteness condition on the Grothendieck category $\cC$ as follows.

Recall that the injective spectrum $\mathbf{Sp}(\cC)$ of \emph{any} Grothendieck category $\cC$ is the set of isomorphism classes of indecomposable injective objects of $\cC$ (cf.\ \cite{Pop} p.\ 331). If $\cL$ is a localizing subcategory of $\cC$, then an object of $\cC$ is said to be \emph{$\cL$-torsion-free} if it has no non-zero subobjects in $\cL$, and the subset $A(\cL) \subseteq \mathbf{Sp}(\cC)$ is by definition the set of isomorphism classes of indecomposable $\cL$-torsion-free injective objects. A subset $A \subseteq \mathbf{Sp}(\cC)$ will be called \emph{stable} if it is of the form $A(\cL)$ for some stable localizing subcategory $\cL$ of $\cC$.

Recall that $\cC$ is said to be \emph{locally noetherian} if it admits a set of noetherian generators. The work of Gabriel \cite{Gab} implies that when $\cC$ is locally noetherian, the map $\cL \mapsto A(\cL)$ gives a bijection between the set $\mathbf{L}^{st}(\cC)$ and the set of stable subsets of $\mathbf{Sp}(\cC)$. It is easy to see that the stable subsets form a topology on $\mathbf{Sp}(\cC)$ that we call the \emph{stable topology}. For any stable subset $A \subseteq \mathbf{Sp}(\cC)$, let $\cL_A$ denote the corresponding stable localizing subcategory of $\cC$ such that $A(\cL_A) = A$. The map $A \mapsto \cL_A$ is inclusion reversing. Hence, for stable subsets $A_1 \subseteq A_2$, we have a quotient functor $\cC/\cL_{A_2} \rightarrow \cC/\cL_{A_1}$. This means that the rule $A \mapsto \cZ_\cC(A) := Z(\cC/\cL_A)$ defines a presheaf $\cZ_\cC$ on $\mathbf{Sp}(\cC)$.

Our second main result is then the following

\begin{theorem}\label{thm:intro-ln-sheaf} Suppose that $\cC$ is a locally noetherian Grothendieck category. Then $\cZ_\cC$ is a sheaf on $\mathbf{Sp}(\cC)$ for the stable topology.
\end{theorem}
Equivalently, the presheaf $Z_\cC$ appearing in Thm. \ref{thm:intro-central-sheaf} satisfies the sheaf condition for arbitrary coverings under the assumption that $\cC$ is locally noetherian.

In the case where $\cC = \Mod(R)$ for some non-commutative noetherian ring $R$, a sheaf of non-commutative rings on $\mathbf{Sp}(\cC)$ in the stable topology is constructed in \cite{GM} and \cite{Lou}. Our sheaf  $\cZ_{\cC}$ is isomorphic to the center of their sheaf. But the techniques in these papers rely heavily on the theory of rings of quotients (\cite{Ste}) and do not generalise to the case of a general locally noetherian category.

We will show elsewhere that $\cC = \Mod(G)$ is a locally noetherian category when $G$ is the group $SL_2(\mathbb{Q}_p)$. We will also explain how to calculate $\mathbf{Sp}(\cC)$ and give a uniform construction of a large family of stable subsets of $\mathbf{Sp}(\cC)$. The resulting sheaf $\cZ_{\cC}$ is then closely related to the recent work of Dotto-Emerton-Gee \cite{DEG}.

In fact, in a lecture in Dublin in July 2019 describing the forthcoming work \cite{DEG}, Emerton talked about the category $\Mod_\chi(G)$ of mod-$p$ smooth $G = GL_2(\mathbb{Q}_p)$-representations with a fixed central character $\chi$. He made the very inspiring suggestion that the center of this category is small and uninteresting, but that this only reflects the fact that this center is the ring of global sections of an interesting sheaf formed by the centers of the quotient categories of $\Mod_\chi(G)$. In our paper \cite{AS} we showed that, indeed, the center of the category $\Mod(G)$ of mod-$p$ smooth $G$-representations of any connected algebraic $p$-adic group $G$ only depends on the center of $G$. The present paper now sets up a general formalism of central sheaves.

The authors acknowledge support from Deutsche Forschungsgemeinschaft (DFG, German Research Foundation) under Germany's Excellence Strategy EXC 2044 –390685587, Mathematics Münster: Dynamics–Geometry–Structure. The second author also acknowledges support from the Mathematical Institute and Brasenose College, Oxford.

\section{The setting}

Let $\fA$ be any abelian category. Its center $Z(\fA)$ is, by definition, the endomorphism ring of the identity functor on $\fA$. It is a commutative ring, which has the following partial functoriality properties:
\begin{itemize}
  \item[--] For any full abelian subcategory $\fA'$ of $\fA$ restriction induces a ring homomorphism $Z(\fA) \rightarrow Z(\fA')$.
  \item[--] Any quotient functor $\fA \rightarrow \fA/\fB$ with respect to a Serre subcategory $\fB$ of $\fA$ induces a ring homomorphism $Z(\fA) \rightarrow Z(\fA/\fB)$.
\end{itemize}

Suppose now that $\cC$ is a Grothendieck category. This means that $\cC$ is an abelian category which has small coproducts (hence small colimits) such that filtered colimits are exact and which has a generator. Such a category is locally small (i.e., the subobjects of any given object form a set) and has enough injective objects (hence injective hulls) as well as small products. In this case we may consider the following construction.

Recall that a full subcategory $\cL$ of $\cC$ is called \emph{localizing} if it is closed under the formation of subobjects, quotient objects, extensions, and arbitrary direct sums. In particular, it is strictly full and abelian.

We note that the collection of all localizing subcategories of $\cC$ is a set (\cite{Pop} Note 3 on p.\ 298), which we will denote by $\mathbf{L}(\cC)$.

Let $\cL$ be a localizing subcategory of $\cC$. Since it is, in particular, a Serre subcategory we may form the quotient category $\cC/\cL$ together with the exact quotient functor $T : \cC \rightarrow \cC/\cL$ (\cite{Gab} Prop.\ 1 on p.\ 367). The functor $T$ has a right adjoint $S : \cC/\cL \rightarrow \cC$, called the \textit{section functor}, which is left exact and fully faithful (\cite{Gab} Prop.\ 8 on p.\ 377 and Prop.\ 2 on p.\ 369 and \cite{Pop} p.\ 177). We also have, for any object $Y$ in $\cC$, a largest subobject $t_\cL(Y)$ of $Y$ which is maximal among all subobjects of $Y$ which lie in $\cL$ (\cite{Gab} Cor.\ 1 on p.\ 375).

\begin{lemma}\phantomsection\label{lem:quotGroth}
\begin{enumerate}
\item $\cL$ and $\cC / \cL$ are both Grothendieck categories, and the quotient functor $T : \cC \rightarrow \cC / \cL$ commutes with inductive limits.
\item $T \circ S \cong \id_{\cC/\cL}$.
\item The section functor $S : \cC / \cL \rightarrow \cC$ preserves injective objects as well as essential extensions.
\end{enumerate}
\end{lemma}
\begin{proof}
(1) is \cite{Gab} Prop.\ 9 on p.\ 378. (2) is \cite{Gab} Prop.\ 3(a) on p.\ 371. (3) The preservation of injective objects follows from the exactness of $T$ and the adjointness $\Hom_{\cC/\cL}(T(-),I) = \Hom_{\cC}(-,S(I))$. For the preservation of essential extensions see \cite{Pop} Cor.\ 4.4.7.
\end{proof}

\begin{remark}\label{rem:iso}
Let $\cL \subseteq \cL'$ be localizing subcategories of $\cC$; we have:
\begin{itemize}
\item[i.] $\cL'/\cL$ is a localizing subcategory of $\cC/\cL$;
\item[ii.] the natural functor $\cC/\cL \rightarrow \cC/\cL'$ induces an equivalence $(\cC/\cL) \big/ (\cL'/\cL) \xrightarrow{\cong} \cC/\cL'$;
\item[iii.] The section functor $S : \cC/\cL \rightarrow \cC$ sends $\cL'/\cL$ to $\cL'$.
\end{itemize}
\end{remark}
\begin{proof}
That $\cL'/\cL$ is a Serre subcategory of $\cC/\cL$ together with the asserted equivalence is a special case of \cite{Pop} Ex.\ 4.3.6. By Lemma \ref{lem:quotGroth} the functor $T$ commutes with arbitrary direct sums. Any family of objects in $\cL'/\cL$ can be lifted to a family in $\cL'$. Since $\cL'$ is closed under the formation of direct sums (in $\cC$) we conclude that $\cL'/\cL$ is closed under the formation of direct sums (in $\cC/\cL$). It follows (\cite{Gab} Prop.\ 8 on p.\ 377 or \cite{Pop} Prop.\ 4.6.3) that $\cL'/\cL$ is localizing in $\cC/\cL$.

For iii. let $T(X)$ be any object in $\cL'/\cL$ for some object $X$ in $\cL'$, and consider the exact sequence $0 \rightarrow t_{\cL}(X) \rightarrow X \rightarrow ST(X) \rightarrow C \rightarrow 0$. The outer terms $t_{\cL}(X)$ and $C$ lie in $\cL$ and hence also in $\cL'$ because $\cL \subseteq \cL'$. Applying the exact quotient functor $T' : \cC \rightarrow \cC/\cL'$ to this exact sequence shows that $T'(ST(X)) \cong T'(X) = 0$. Hence $S(T(X))$ lies in $\ker(T') = \cL'$ as required.
\end{proof}

Remark \ref{rem:iso} says that for localizing subcategories $\cL \subseteq \cL'$ of $\cC$ the natural functor $\cC/\cL \rightarrow \cC/\cL'$ is a quotient functor and hence induces a natural ring homomorphism $Z(\cC/\cL) \rightarrow Z(\cC/\cL')$ between the corresponding centers. Instead of just considering the center $Z(\cC)$ we propose to investigate the ``presheaf'' of commutative rings $\cL \mapsto Z(\cC/\cL)$ on $\mathbf{L}(\cC)$. This will require the concept of stability of $\cL$, which we will discuss in the next section.

Here we only remark that obviously the intersection of any family of localizing subcategories again is localizing. Therefore, for any two localizing subcategories $\cL_1, \cL_2$ of $\cC$ the smallest localizing subcategory $\cL_1 \vee \cL_2$ of $\cC$ which contains $\cL_1$ and $\cL_2$ is well defined.

\section{Stability}\label{sec:stability}

Throughout the paper we let $\cC$ denote a Grothendieck category. A localizing subcategory $\cL$ of $\cC$ is called \textit{stable} if it is closed under the passage to essential extensions. Obviously the intersection of any family of stable localizing subcategories is stable.

\begin{lemma}\label{lem:stable-T-inj}
 For any stable localizing subcategory $\cL$ of $\cC$ the quotient functor $T : \cC \rightarrow \cC / \cL$ preserves injective objects.
\end{lemma}
\begin{proof}
\cite{Gab} Cor.\ 3 on p.\ 375.
\end{proof}

\begin{proposition}\label{prop: Gab375Cor2}
   Suppose that $\cL$ is a stable localizing subcategory of $\cC$ and let $Y$ be an injective object in $\cC$. We have:
\begin{enumerate}
\item $ST(Y)$ is injective in $\cC$.
\item There is an isomorphism $Y\cong t_{\cL}(Y) \oplus ST(Y)$ in $\cC$.
\item $t_{\cL}(Y)$ is injective in $\cC$.
\end{enumerate}
\end{proposition}
\begin{proof} (1) This follows from Lemma \ref{lem:stable-T-inj} and Lemma \ref{lem:quotGroth}(3).

(2) This follows from \cite{Gab} Cor.\ 2 on p.\ 375, using the stability assumption on $\cL$.

(3) By (2), $t_{\cL}(Y)$ is a direct summand of the injective object $Y$.
\end{proof}

\begin{lemma}\label{lem: 3IT}
Let $\cL \subseteq \cL'$ be localizing subcategories of $\cC$ such that $\cL$ is stable. Then $\cL'$ is stable in $\cC$ if and only if $\cL'/\cL$ is stable in $\cC/\cL$.
\end{lemma}
\begin{proof}
``$\Longrightarrow$'' Let $X \hookrightarrow Y$ be an essential extension in $\cC/\cL$, where $X$ is an object in $\cL'/\cL$. Then $S(X) \hookrightarrow S(Y)$ is an essential extension in $\cC$ by Lemma \ref{lem:quotGroth}(3). Since $S(X)$ lies in $\cL'$ by Remark \ref{rem:iso}.iii and since $\cL'$ is stable in $\cC$ by assumption, we see that $S(Y)$ lies in $\cL'$. Using Lemma \ref{lem:quotGroth}(1), we deduce that $Y \cong TS(Y)$ lies in $\cL'/\cL$ as required.

``$\Longleftarrow$'' Let $X$ be a non-zero object in $\cL'$ and let $Y$ be an injective hull of $X$ in $\cC$; we have to show that $Y$ lies in $\cL'$. Since $\cL$ is stable, using Prop.\ \ref{prop: Gab375Cor2}(2) we find an isomorphism $Y  \cong W \oplus Y'$ where $W := t_{\cL}(Y)$ and $Y' := ST(Y)$. We will show that $Y'$ lies in $\cL'$. We may assume that $Y' \neq 0$, as the claim is clear otherwise.

Consider the essential extension $j : X \hookrightarrow W \oplus Y'$. Since $Y'$ is non-zero, $X' := j(X) \cap Y'$ is essential in $Y'$. Since $Y'$ is $\cL$-torsion-free, $T(X')$ is essential in $T(Y')$ by \cite{Gab} Prop.\ 6 on p.\ 374. Since $X'$ is a subobject of $X$, $X'$ lies in the localizing subcategory $\cL'$, so $T(X')$ lies in $\cL'/\cL$. Since $\cL'/\cL$ is stable in $\cC/\cL$ by assumption, the essential extension $T(Y')$ of $T(X')$ lies in $\cL'/\cL$ as well. Hence $ST(Y')$ lies in $\cL'$ by Remark \ref{rem:iso}.iii. However $ST(Y') = STST(Y) \cong ST(Y) = Y'$, so $Y'$ lies in $\cL'$ as claimed.

Since $W$ lies in $\cL$ and since $\cL \subseteq \cL'$, we see that $Y \cong W \oplus Y'$ lies in $\cL'$ as required.
\end{proof}

Next we will establish that with $\cL_1$ and $\cL_2$ also $\cL_1 \vee \cL_2$ is stable.

\begin{proposition}\label{prop: stableL3}
  Let $\cL_1$ and $\cL_2$ be two localizing subcatgories of $\cC$ such that $\cL_1 \cap \cL_2 = 0$. The strictly full subcategory $\cL_{1,2}$ of $\cC$ whose objects are isomorphic to $X_1 \oplus X_2$ with $X_i$ being an object of $\cL_i$ for $i = 1, 2$ is a stable localizing subcategory of $\cC$.
\end{proposition}
\begin{proof}
Let $X = X_1 \oplus X_2$ be an object in $\cL_{1,2}$, where $X_i$ lies in $\cL_i$ for $i=1,2$. Let $Y$ be a subobject of $X$ in $\cC$ and let $Y_i = X_i \cap Y$ for $i = 1,2$. Then $Y_1 \oplus Y_2$ is an object in $\cL_{1,2}$ contained in $Y$. We have the monomorphism
\begin{equation*}
  Y/Y_1 = Y/(X_1 \cap Y) \cong (Y + X_1)/X_1 \hookrightarrow X/X_1 \cong X_2 \ ,
\end{equation*}
which shows that $Y/Y_1$ lies in $\cL_2$. On the other hand $Y/Y_1$ admits an epimorphism onto $Y / (Y_1 \oplus Y_2)$ so that the latter lies in $\cL_2$ as well. The same argument works with the roles of the indices $1$ and $2$ exchanged. We deduce that $Y / (Y_1 \oplus Y_2)$ lies in $\cL_1 \cap \cL_2$, and is therefore zero. Hence $Y = Y_1 \oplus Y_2$ lies in $\cL_{1,2}$. By passing to the quotient objects, we see that $X/Y \cong (X_1/Y_1) \oplus (X_2/Y_2)$ also lies in $\cL_{1,2}$.

We have shown that $\cL_{1,2}$ is closed under quotient objects and subobjects. Since the formation of injective hulls commutes with finite direct sums by \cite{Gab} Lemma 2 on p.\ 358, and since $\cL_1$ and $\cL_2$ are both stable in $\cC$, it is clear that $\cL_{1,2}$ is stable under injective hulls in $\cC$. Since $\cL_3$ is clearly closed under the formation of small coproducts in $\cC$, we see that it remains to show that $\cL_{1,2}$ is closed under extensions in $\cC$.

Let $0 \rightarrow X \rightarrow Y \rightarrow Z \rightarrow 0$ be a short exact sequence in $\cC$, where $X$ and $Z$ are objects in $\cL_{1,2}$. Let $I$ and $J$ be injective envelopes of $X$ and $Z$, respectively. Using (the dual of) the Horseshoe Lemma (\cite{Wei} Lemma 2.2.8) we find a commutative diagram in $\cC$
\begin{equation*}
  \xymatrix{
   & 0 \ar[d] & 0 \ar[d] & 0 \ar[d] & \\ 0 \ar[r] & X \ar[d]\ar[r] & Y \ar[r]\ar[d] & Z \ar[r]\ar[d] & 0 \\ 0 \ar[r] & I \ar[r] & I \oplus J \ar[r] & J \ar[r] & 0 }
\end{equation*}
with exact rows and columns. We saw above that $\cL_{1,2}$ is stable under injective envelopes in $\cC$, so $I$ and $J$ lie in $\cL_{1,2}$ because $X$ and $Z$ lie in $\cL_{1,2}$. Hence $I \oplus J$ also lies in $\cL_{1,2}$ and its subobject $Y$ lies in $\cL_{1,2}$ as well by the first paragraph above.
\end{proof}

\begin{corollary}\label{cor:stableL3}
   Suppose that $\cL_1 \cap \cL_2 = 0$. Then $\cL_1 \vee \cL_2 = \cL_{1,2}$, and $\cL_1 \vee \cL_2$ is stable in $\cC$.
\end{corollary}
\begin{proof}
Clearly $\cL_{1,2}$ must be contained in any localizing subcategory of $\cC$ containing both $\cL_1$ and $\cL_2$, so $\cL_{1,2}$ is contained in $\cL_1 \vee \cL_2$. Conversely, $\cL_{1,2}$ \emph{is} a localizing subcategory of $\cC$ by Prop.\ \ref{prop: stableL3}, so $\cL_1 \vee\cL_2$ is contained in $\cL_{1,2}$. Hence $\cL_1 \vee \cL_2 = \cL_{1,2}$ is stable in $\cC$ by Prop.\ \ref{prop: stableL3}.
\end{proof}

\begin{corollary}\label{cor: t123}
  If $\cL_1 \cap \cL_2 = 0$ then $t_{\cL_1 \vee \cL_2} = t_{\cL_1} \oplus\, t_{\cL_2}$ and $t_{\cL_1} \circ t_{\cL_2} = t_{\cL_2} \circ t_{\cL_1} = 0$.
\end{corollary}
\begin{proof}
Let $X$ be an object in $\cC$. Note first that $t_{\cL_1}(X) \cap t_{\cL_2}(X)$ lies in $\cL_1 \cap \cL_2$ and is therefore zero. Hence the sum $t_{\cL_1}(X) + t_{\cL_2}(X)$ is direct, and is therefore contained in $t_{\cL_1 \vee \cL_2}(X)$. On the other hand, $t_{\cL_1 \vee \cL_2}(X)$ is an object in $\cL_1 \vee \cL_2$, so it is of the form $X_1 \oplus X_2$ with $X_i$ lying in $\cL_i$ for $i=1,2$. Then $X_i \leq t_{\cL_i}(X)$ for $i=1,2$, so $t_{\cL_1 \vee \cL_2}(X) = X_1 \oplus X_2$ is contained in $(t_{\cL_1} \oplus t_{\cL_2})(X)$. Finally $t_{\cL_1}(t_{\cL_2}(X))$ and $t_{\cL_2}(t_{\cL_1}(X))$ both lie in $\cL_1 \cap \cL_2$ and are therefore zero.
\end{proof}

\begin{proposition}\label{prop:vee-stable}
   If $\cL_1$ and $\cL_2$ are stable localizing subcategories of $\cC$ then $\cL_1 \vee \cL_2$ is stable as well.
\end{proposition}
\begin{proof}
Since obviously with $\cL_1$ and $\cL_2$ also $\cL_1 \cap \cL_2$ is stable in $\cC$ we use Lemma \ref{lem: 3IT} to confirm that $\cL_1/(\cL_1 \cap \cL_2)$ and $\cL_2/(\cL_1 \cap \cL_2)$ are stable in $\cC/(\cL_1 \cap \cL_2)$. Their intersection is zero so that $(\cL_1/(\cL_1 \cap \cL_2)) \vee (\cL_2/(\cL_1 \cap \cL_2))$ is stable in $\cC/(\cL_1 \cap \cL_2)$ by Cor.\ \ref{cor:stableL3}. Using \cite{NT} Prop.\ 3.2 one checks that $(\cL_1/(\cL_1 \cap \cL_2)) \vee (\cL_2/(\cL_1 \cap \cL_2)) = (\cL_1 \vee \cL_2)/(\cL_1 \cap \cL_2)$. The reverse direction in Lemma \ref{lem: 3IT} finally implies that $\cL_1 \vee \cL_2$ is stable in $\cC$.
\end{proof}

\section{Identifying the Gabriel gluing}\label{sec:gluing}

We fix two localizing subcategories $\cL_0, \cL_1$ of $\cC$. We then have the following natural commutative diagram of Grothendieck categories and exact quotient functors\footnote{For simplicity we take the point of view that quotient categories have the same objects as the categories they are the quotient of.}:
\begin{equation*}
  \xymatrix{
   & \cC / \cL_0 \ar[rd]^{F_0} & \\ \cC/(\cL_0 \cap \cL_1) \ar[ur]^{Q_0}\ar[dr]_{Q_1} & & \cC / (\cL_0 \vee \cL_1) \\ & \cC/\cL_1 \ar[ur]_{F_1} &
   }
\end{equation*}
In this situation \cite{Gab} \S IV.1 on p.\ 439 introduces the \emph{recollement} category
\begin{equation*}
  \cD := (\cC/\cL_0) \underset{\cC / (\cL_0 \vee \cL_1)}{\prod}{} (\cC/\cL_1)
\end{equation*}
In addition \cite{Gab} Prop.\ 1 on p.\ 440 gives us a natural functor
\begin{equation*}
  R : \cC/(\cL_0 \cap \cL_1) \longrightarrow  (\cC/\cL_0) \underset{\cC / (\cL_0 \vee \cL_1)}{\prod}{} (\cC/\cL_1)
\end{equation*}
which fits in the above diagram as follows:
\begin{equation}\label{diag:gluing}
  \xymatrix{
   & & \cC / \cL_0 \ar[rd]^{F_0} & \\ \cC/(\cL_0 \cap \cL_1) \ar[urr]^{Q_0} \ar[drr]_{Q_1} \ar[r]^-{R}  & (\cC/\cL_0) \underset{\cC / (\cL_0 \vee \cL_1)}{\prod}{} (\cC/\cL_1) \ar[ur]_(0.6){T_0} \ar[dr]^(0.6){T_1}& &   \cC / (\cL_0 \vee \cL_1) \\ & &  \cC/\cL_1 \ar[ur]_{F_1} &
   }
\end{equation}

\begin{lemma}\label{lem:Rfaithful}
The functor $R$ is exact and faithful.
\end{lemma}
\begin{proof}
Given an object $X$ in $\cC/(\cL_0 \cap \cL_1)$, let $\eta_X : F_0 Q_0 (X) \rightarrow F_1 Q_1 (X)$ be the identity map of $X$ viewed in $\cC / (\cL_0 \vee \cL_1)$. Then the functor $R$ is given on objects by $R(X) = (Q_0(X), Q_1(X), \eta_X)$. It follows immediately from the exactness of $Q_0$ and $Q_1$ that $R$ is also exact.

Since $R$ is additive, to show that $R$ is a faithful functor, it is enough to show that $R(f) = 0$ implies $f = 0$. Suppose, then, that $f : X \rightarrow Y$ is such that $R(f) : R(X) \rightarrow R(Y)$ is zero in the recollement category $\cD$. Then the projections of $R(f)$ to $\cC/\cL_0$ and $\cC/\cL_1$ are zero, so both $Q_0(f) : Q_0(X) \rightarrow Q_0(Y)$  and $Q_1(f) : Q_1(X) \rightarrow Q_1(Y)$  are zero. This means that $\im(f)$ lies in $\cL_1/(\cL_0 \cap \cL_1)$ as well as in $\cL_1/(\cL_0 \cap \cL_1)$. So $\im(f) = 0$ and hence $f = 0$. We see that $R$ is faithful.
\end{proof}

All the functors $Q_i$, $F_i$ (by Remark \ref{rem:iso}) and $T_i$ (by \cite{Gab} Lemma 1 on p.\ 442) are quotient functors and hence are exact and have right adjoint section functors $U_i$, $H_i$, and $S_i$ respectively. For the convenience of the reader this is depicted in the following completed diagram:
\begin{equation}\label{diag:gluing2}
  \xymatrix{
   & & \cC / \cL_0 \ar@<-0.5ex>[lldd]^-{U_0} \ar@<-0.5ex>[ldd]^-{S_0} \ar@<1ex>[rdd]^-{F_0} & \\
   &&&     \\
   \cC/(\cL_0 \cap \cL_1) \ar@<1ex>[uurr]^-{Q_0} \ar@<-0.5ex>[ddrr]^-{Q_1}  \ar[r]^-R  & (\cC/\cL_0) \underset{\cC / (\cL_0 \vee \cL_1)}{\prod}{} (\cC/\cL_1) \ar@<1ex>[uur]^-{T_0} \ar@<1ex>[ddr]^-{T_1} & &   \cC / (\cL_0 \vee \cL_1)
   \ar@<-0.5ex>[luu]^-{H_0} \ar@<-0.5ex>[ldd]^{H_1} \\
   &&&    \\
   & &  \cC/\cL_1 \ar@<1ex>[uull]^-{U_1} \ar@<-0.5ex>[luu]^-{S_1} \ar@<1ex>[uur]^-{F_1} &   }
\end{equation}
(We will not use the functors $S_i$, though.)

These section functors allow us to construct a functor
\begin{equation*}
  V : (\cC/\cL_0) \underset{\cC / (\cL_0 \vee \cL_1)}{\prod}{} (\cC/\cL_1) \rightarrow \cC/(\cL_0 \cap \cL_1)
\end{equation*}
in the opposite direction. For this we need:
\begin{itemize}
  \item The composite $U_i \circ H_i$ is a section functor for the composite $F_i \circ Q_i$ for $i = 0, 1$. But $F_0 \circ Q_0 = F_1 \circ Q_1$ (by our convention for quotient categories). Hence we have a natural isomorphism $\iota : U_0 \circ H_0 \xrightarrow{\cong} U_1 \circ H_1$.
  \item Let $\eta_i : \id_{\cC/\cL_i} \rightarrow H_i \circ F_i$ and $\varepsilon_i : F_i \circ H_i \xrightarrow{\cong} \id_{\cC/(\cL_i \vee \cL_{1-i})}$, for $i = 0, 1$, denote the unit and the counit of the adjunction between $F_i$ and $H_i$, respectively.
  \item Let $\delta_i : \id_{\cC/(\cL_i \cap \cL_{1-i})} \rightarrow U_i \circ Q_i$ and $\gamma_i : Q_i \circ U_i \xrightarrow{\cong} \id_{\cC/\cL_i}$, for $i = 0, 1$, denote the unit and the counit of the adjunction between $Q_i$ and $U_i$, respectively.
\end{itemize}
Let now $\Theta = (X_0,X_1,\sigma)$ be an object in $\cD$. Recall that $\sigma$ is an isomorphism $F_0(X_0) \xrightarrow[\cong]{\sigma} F_1(X_1)$. Consider now the solid arrow diagram
\begin{equation}\label{diag:V}
  \xymatrix{
    V(X_0,X_1,\sigma) \ar@{-->}[d]_{\pi_{0,\Theta}} \ar@{-->}[rrrr]^{\pi_{1,\Theta}} &  &&  & U_1(X_1)  \ar[d]^{U_1(\eta_{1,X_1})}  \\
    U_0(X_0)  \ar[rr]^-{U_0(\eta_{0,X_0})} && U_0 H_0 F_0 (X_0) \ar[r]^{U_0 H_0 (\sigma)}_{\cong} & U_0 H_0 F_1 (X_1) \ar[r]^{\iota_{F_1(X_1)}}_{\cong} & U_1 H_1 F_1(X_1)       }
\end{equation}
and take its fiber product $V(X_0,X_1,\sigma)$ in $\cC / (\cL_0 \cap \cL_1)$.

In order to investigate the relation between $R$ and $V$ we first have to recall a few facts about adjunctions.

\begin{lemma}\phantomsection\label{lem:adjoints}
\begin{itemize}
  \item[i.] Let $A : \cA \rightarrow \cB$ be a functor and let $B_i : \cB \rightarrow \cA$, for $i = 0, 1$, be two right adjoint functors to $A$ with corresponding counits $\alpha_i : A \circ B_i \rightarrow \id_\cB$ and units $\beta_i : \id_\cA \rightarrow B_i \circ A$. Then there is a unique natural isomorphism $\nu : B_0 \xrightarrow{\cong} B_1$ such that the diagrams
\begin{equation*}
  \xymatrix{
  B_0 \circ A \ar[rr]^{\nu_{A(-)}} &  &  B_1 \circ A     \\
                & \id_\cA \ar[ul]^{\beta_0}   \ar[ur]_{\beta_1}             }
  \qquad\text{and}\qquad
  \xymatrix{
  A \circ B_0 \ar[rr]^{A(\nu)} \ar[dr]_{\alpha_0}
                &  &    A \circ B_1 \ar[dl]^{\alpha_1}    \\
                & \id_\cB                 }
\end{equation*}
      are commutative; it is given by
\begin{equation*}
   \nu_Y = B_1(\alpha_{0,Y}) \circ \beta_{1,B_0(Y)} : B_0(Y) \rightarrow B_1 A B_0(Y) \rightarrow B_1(Y) \qquad\text{for $Y$ in $\cB$} .
\end{equation*}
  \item[ii.] Let $A : \cA \rightarrow \cA'$ be a functor with right adjoint $B : \cA' \rightarrow \cA$ and counit $\alpha$ and unit $\beta$; furthermore, let $A' : \cA' \rightarrow \cA''$ be a functor with right adjoint $B': \cA'' \rightarrow \cA'$ and counit $\alpha'$ and unit $\beta'$. Then $A' \circ A : \cA \rightarrow \cA''$ has the right adjoint $B \circ B' : \cA'' \rightarrow \cA$; the corresponding counit and unit are the composites
\begin{equation*}
  A' A B B'(-) \xrightarrow{A'(\alpha_{B'(-)})} A' B'(-) \xrightarrow{\alpha'_{-}} \id_{\cA''}(-) \ \text{and}\ \id_\cA(-) \xrightarrow{\beta_{-}} BA(-) \xrightarrow{B(\beta'_{A(-)})} B B' A' A(-) ,
\end{equation*}
  respectively.
\end{itemize}
\end{lemma}
\begin{proof}
i. This is well known of course. But for the convenience of the reader we sketch a proof. The two adjunction isomorphisms
\begin{equation*}
  \xymatrix@R=0.5cm{
                &         \Hom_\cA(X,B_0(Y)) \ar[dd]^{\psi_{X,Y}}_{\cong}     \\
  \Hom_\cB(A(X),Y) \ar[ur]^{\adj_0}_{\cong} \ar[dr]_{\adj_1}^{\cong}                 \\
                &         \Hom_\cA(X,B_1(Y))                 }
\end{equation*}
give rise to the natural isomorphism $\psi$. By the Yoneda lemma it has to be of the form $\psi_{X,Y} = \Hom_\cA(X,\nu_Y)$ for a unique natural isomorphism $\nu : B_0 \xrightarrow{\cong} B_1$. The units satisfy
\begin{equation*}
  \beta_{i,X} = \adj_i(\id_{A(X)}) \ .
\end{equation*}
It immediately follows that $\nu_{A(X)} \circ \beta_{0,X} = \beta_{1,X}$ which is the commutativity of the left hand diagram. The counits satisfy
\begin{equation*}
  \adj_i(\alpha_{i,Y}) = \id_{B_i(Y)} \ .
\end{equation*}
By, vice versa, expressing the $\adj_i$ and hence $\psi$ in terms of the units and counits one easily derives the asserted formula for $\nu_Y$.

Finally consider the diagram
\begin{equation*}
  \xymatrix{
    A B_0(Y) \ar@{=}[drr] \ar[rr]^-{A(\beta_{1,B_0(Y)})} && A B_1 A B_0(Y) \ar[d]^{\alpha_{1,AB_0(Y)}} \ar[rr]^-{A B_1(\alpha_{0,Y})} && A B_1(Y) \ar[d]^{\alpha_{1,Y}} \\
              && A B_0(Y) \ar[rr]^-{\alpha_{0,Y}} && Y .   }
\end{equation*}
The left hand part commutes by one of the two triangular identities for the adjunction between $A$ and $B_1$ (\cite{ML} IV.1 Thm.\ 1), the right hand one by the naturality of $\alpha_1$. The composite of the top horizontal arrows is equal to $A(\nu_Y)$ by the formula for $\nu_Y$. This shows the commutativity of the right hand diagram in the assertion.

ii. \cite{ML} IV.8 Thm.\ 1.
\end{proof}

We first consider the solid part fo the diagram \eqref{diag:V} for an object $\Theta$ of the form $\Theta = R(X) = (Q_0(X), Q_1(X), \eta_X)$. Then $U_0H_0(\eta_X)$ is simply the identity morphism. Hence this solid part becomes the solid part of the diagram
\begin{equation}\label{diag:VR}
  \xymatrix{
  X \ar@{-->}[drrrr]^{\delta_{1,X}} \ar@{-->}[ddr]_{\delta_{0,X}} &&&&   \\
   &  &  &  & U_1 Q_1(X)  \ar[d]^{U_1(\eta_{1,Q_1(X)})}  \\
   &  U_0 Q_0(X)  \ar[rr]^-{U_0(\eta_{0,Q_0(X)})} && U_0 H_0 F_0 Q_0(X) = U_0 H_0 F_1 Q_1(X) \ar[r]^-{\iota_{F_1 Q_1(X)}}_-{\cong} & U_1 H_1 F_1Q_1(X)       }
\end{equation}
If $\rho_i$ denotes the unit for the adjunction between $F_1 \circ Q_1$ and $U_i \circ H_i$ then Lemma \ref{lem:adjoints}.ii tells us that $\rho_{i,X} = U_i(\eta_{i,Q_i(X)}) \circ \delta_{i,X}$. Hence the complete diagram \eqref{diag:VR} simplifies to the diagram
\begin{equation*}
  \xymatrix{
                & X \ar[dl]_{\rho_{0,X}} \ar[dr]^{\rho_{1,X}}             \\
 U_0 H_0 F_1 Q_1(X) \ar[rr]^{\iota_{F_1 Q_1(X)}} & & U_1 H_1 F_1 Q_1(X) ,      }
\end{equation*}
which commutes by Lemma \ref{lem:adjoints}.i. By the universal property of the fiber product we therefore obtain a natural transformation
\begin{equation}\label{f:RVunit}
  \id_{\cC/(\cL_0 \cap \cL_1)} \longrightarrow V \circ R \ .
\end{equation}

Next we study the other composite $R \circ V$. But for this we will always \textbf{assume in the following that $\cL_0$ and $\cL_1$ are stable.}

\begin{lemma}\label{lem:U}
   Suppose that $\cL_1 \cap \cL_2 = 0$; we then have:
\begin{itemize}
  \item[i.] $\Ext_\cC^j(X_0, X_1) = 0$ for any $X_0$ in $\cL_0$ and $X_1$ in $\cL_1$ and any $j \geq 0$.
  \item[ii.] For $i = 0, 1$ any object in $\cL_i$ is $\cL_{1-i}$-closed.
  \item[iii.] For $i = 0, 1$ and any $j \geq 0$ the derived functor $R^j U_i$ sends $(\cL_0 \vee \cL_1)/\cL_i$ to $\cL_{1-i}$.
\end{itemize}
\end{lemma}
\begin{proof}
i. Let $f : X_0 \rightarrow X_1$ be a morphism in $\cC$. Then $\im(f)$ is a quotient object of $X_0$, so it lies in $\cL_0$; on the other hand it is a subobject of $X_1$ so it lies in $\cL_1$. Hence $\im(f)$ lies in $\cL_0 \cap \cL_1$ and is therefore zero. Hence $f = 0$, so $\Hom_{\cC}(X_0,X_1) = 0$.

Choose an injective resolution $X_1 \xrightarrow{\simeq} I^\bullet$ in $\cC$. Since $X_1$ lies in $\cL_1$ which is stable in $\cC$, we see that $I^j$ lies in $\cL_1$ for all $j \geq 0$, so $\Hom_{\cC}(X_0,I^j) = 0$ for all $j \geq 0$ by the first paragraph. Therefore $\Ext^j_{\cC}(X_0,X_1) = H^j(\Hom_{\cC}(X_0, I^\bullet)) = 0$ as well.

ii. Let $X$ be an object in $\cL_i$, so $X = \mathfrak{t}_{\cL_i}(X)$. Then $\mathfrak{t}_{\cL_{1-i}}(X) = \mathfrak{t}_{\cL_{1-i}}(\mathfrak{t}_{\cL_i}(X)) = 0$ by Cor.\ \ref{cor: t123}, so $X$ is $\cL_{1-i}$-torsion-free. Next, any short exact sequence $0 \rightarrow X \rightarrow Y \rightarrow Z \rightarrow 0$ with $Z$ lying in $\cL_{1-i}$ must split, because $\Ext^1_{\cC}(Z, X) = 0$ by i. Hence $X$ is $\cL_{1-i}$-closed by \cite{Gab} Lemma 1.b on p.\ 370.

iii. Every object $Y$ of $(\cL_0 \vee \cL_1)/\cL_i$ is of the form $Y = Q_i(X_i \oplus X_{1-i})$, where $X_j$ lies in $\cL_j$ for $j=0,1$. Since $Q_i (X_i) = 0$, we see that  $U_i (Y) = U_i Q_i (X_{1-i})$. However $X_{1-i}$ lies in $\cL_{1-i}$, so it is $\cL_i$-closed by ii. Therefore $U_i Q_i (X_{1-i}) \cong X_{1-i}$, so $U_i(Y) \cong X_{1-i}$ lies in $\cL_{1-i}$.

Now pick an injective resolution $Y \xrightarrow{\simeq} I^\bullet$ in $\cC/\cL_i$. From Cor.\ \ref{cor:stableL3} we know that $\cL_0 \vee \cL_1$ is stable in $\cC$. Hence $(\cL_0 \vee \cL_1)/\cL_i$  is stable in $\cC/\cL_i$ by Lemma \ref{lem: 3IT}. It follows that the resolution $I^\bullet$ lies, in fact, in $(\cL_0 \vee \cL_1)/\cL_i$. Applying now the first paragraph with each $I^j$ the assertion easily follows.
\end{proof}

\begin{lemma}\label{lem:Ugeneral}
   For $i = 0, 1$ and any $j \geq 0$ the derived functor $R^j U_i$ sends $(\cL_0 \vee \cL_1)/\cL_i$ to $\cL_{1-i}/(\cL_0 \cap \cL_1)$.
\end{lemma}
\begin{proof}
Write $\cN := \cL_0 \cap \cL_1$. By the universal property of quotient categories we have the following diagram of exact quotient functors:
\begin{equation*}
  \xymatrix{
  \cC \ar[rrr] \ar[d] &&& \cC/\cN \ar[d]^{Q_i'}\ar[dlll]_-{Q_i} \\
  \cC/\cL_i  &&& (\cC/\cN) / (\cL_i/\cN) \ar[lll]^-{B_i} }
\end{equation*}
Note that $Q_i$ is the functor  appearing in the basic diagram \eqref{diag:gluing}. By Remark \ref{rem:iso} the functor $B_i$ is an equivalence. Fix a quasi-inverse $A_i : \cC/\cL_i \xrightarrow{\simeq} (\cC/\cN) / (\cL_i/\cN)$ to $B_i$. Now, $A_iQ_i = A_i B_i Q_i' \cong Q_i'$. Since $B_i$ is a right adjoint to $A_i$ and $U_i$ is a right adjoint to $Q_i$, the functor $U_i' := U_i B_i$ is a right adjoint to $A_i Q_i \cong Q_i'$.

We have to show that $R^j( Q_{1-i} U_i )$ kills $(\cL_0 \vee \cL_1)/\cL_i$ for all $j \geq 0$. This has already been shown in the case where $\cN = 0$ in Lemma \ref{lem:U}.iii. Fix $j \geq 0$. Now, $\cL_0/\cN$ and $\cL_1/\cN$ are stable localizing subcategories of $\cC/\cN$ by Lemma \ref{lem: 3IT}, and their intersection is zero. Hence $R^j( Q_{1-i}' U_i')$ kills $\frac{(\cL_0/\cN) \vee (\cL_1/\cN)}{\cL_i/\cN}$ by this special case.

We have $Q_{1-i}' \cong A_{1-i} Q_{1-i}$ and $U_i' = U_i B_i$. Since $A_{1-i}$ and $B_i$ are equivalences of categories, we deduce that $R^j(Q_{1-i} U_i) B_i$ kills $\frac{(\cL_0/\cN) \vee (\cL_1/\cN)}{\cL_i/\cN}$. As noted already in the proof of Prop.\ \ref{prop:vee-stable}, we have $(\cL_0/\cN) \vee (\cL_1/\cN) = (\cL_0 \vee \cL_1)/\cN$. Hence $R^j(Q_{1-i} U_i)B_i$ kills $\frac{(\cL_0 \vee \cL_1)/\cN}{\cL_i/\cN}$.

Let $X$ be an object of $(\cL_0 \vee \cL_1)/\cL_i$. Then $X = Q_i (Y)$ for some object $Y$ of $(\cL_0 \vee \cL_1)/\cN$. But $Q_i = B_i Q_i'$, so $X = B_i Q_i'(Y)$, where $Q_i'(Y)$ lies in $\frac{(\cL_0 \vee \cL_1)/\cN}{\cL_i/\cN}$. Therefore $R^j(Q_{1-i} U_i) (X) = R^j(Q_{1-i} U_i)B_i Q_i' (Y) = 0$ by the above.
\end{proof}

We claim that the map $Q_0(\pi_{0,\Theta}) : Q_0 V(X_0, X_1, \sigma) \xrightarrow{\cong} Q_0 U_0 (X_0)$ is an isomorphism. Since the functor $Q_0$ is exact applying it to the diagram \eqref{diag:V} results in another fiber product diagram. Hence it is enough to show that the map $Q_0 U_1(\eta_{1,X_1}) : Q_1 U_1(X_1) \rightarrow Q_0 U_1 H_1 F_1(X_1)$ is an isomorphism. By \cite{Gab} Prop.\ 3.b on p.\ 371 kernel and cokernel of the map $\eta_{1,X_1}$ lie in $\ker(F_1) = (\cL_0 \vee \cL_1)/\cL_1$. On the other hand, by Lemma \ref{lem:Ugeneral}, the derived functors $Q_0 R^j U_1 = R^j(Q_0 U_1)$, for any $j \geq 0$, are zero on $(\cL_0 \vee \cL_1)/\cL_1$. It easily follows that $Q_0 U_1(\eta_{1,X_1})$ is an isomorphism. By symmetry we have that $Q_1(\pi_{1,\Theta}) : Q_1 V(X_0, X_1, \sigma) \xrightarrow{\cong} Q_1 U_1 (X_1)$ is an isomorphism as well.

\begin{proposition}\label{prop:R-equiv}
   Suppose that $\cL_0$ and $\cL_1$ are stable. Then
\begin{equation*}
  R V(X_0,X_1,\sigma) \xrightarrow{(\gamma_{0,X_0} \circ Q_0(\pi_{0,\Theta}), \gamma_{1,X_1} \circ Q_1(\pi_{1,\Theta}))} (X_0,X_1,\sigma)   \qquad\text{for $\Theta = (X_0,X_1,\sigma)$}
\end{equation*}
is a natural isomorphism between functors on $\cD$, and $R: \cC/(\cL_0 \cap \cL_1) \xrightarrow{\simeq} \cD$ is an equivalence of categories.
\end{proposition}
\begin{proof}
For the first part of the assertion it remains to show that the diagram
\begin{equation*}
  \xymatrix{
    F_0 Q_0 V(\Theta) \ar[d]_{\eta_{V(\Theta)} = \id} \ar[rrr]^-{F_0(\gamma_{0,X_0} \circ Q_0(\pi_{0,\Theta}))} &&& F_0(X_0) \ar[d]^{\sigma} \\
    F_1 Q_1 V(\Theta) \ar[rrr]^-{F_1(\gamma_{1,X_1} \circ Q_1(\pi_{1,\Theta}))} &&& F_1(X_1)   }
\end{equation*}
is commutative. Consider the commutative diagram:
\begin{equation*}
  \xymatrix{
    Q_0 V(X_0,X_1,\sigma) \ar[d]_{Q_0(\pi_{0,\Theta})} \ar[rrr]^{Q_0(\pi_{1,\Theta})} &  &  & Q_0 U_1(X_1)  \ar[d]^{Q_0 U_1(\eta_{1,X_1})}  \\
    Q_0 U_0(X_0)\  \ar[d]_{\gamma_{0,X_0}}^{\cong}  \ar[r]^-{Q_0 U_0(\eta_{0,X_0})} & \  Q_0 U_0 H_0 F_0 (X_0) \ \ar[d]_{\gamma_{0,H_0 F_0(X_0)}}^{\cong} \ar[r]^{Q_0 U_0 H_0 (\sigma)}_{\cong} & \  Q_0 U_0 H_0 F_1 (X_1) \ \ar[d]_{\gamma_{0,H_0 F_1(X_1)}}^{\cong} \ar[r]^{Q_0(\iota_{F_1(X_1)})}_{\cong} & \ Q_0 U_1 H_1 F_1(X_1)    \\
    X_0 \ar[r]^-{\eta_{0,X_0}} & \quad H_0 F_0(X_0) \ar[r]^{H_0(\sigma)}_{\cong} & H_0 F_1(X_1)  }
\end{equation*}
The upper square is $Q_0$ applied to the diagram \eqref{diag:V}, which is still a fiber product diagram. The lower two squares are commutative by the naturality of the counit $\gamma_0$. We now apply the functor $F_0$. First notice that the bottom line becomes the top line of the diagram
\begin{equation*}
  \xymatrix{
    F_0(X_0) \ \ar@{=}[d] \ar[r]^-{F_0(\eta_{0,X_0})} & \ F_0 H_0 F_0(X_0) \ \ar[r]^{F_0 H_0(\sigma)}_{\cong} \ar[d]_{\varepsilon_{0,F_0(X_0)}}^{\cong} & \ F_0 H_0 F_1(X_1) \ar[d]_{\varepsilon_{0,F_1(X_1)}}^{\cong} \\
    F_0(X_0) \ar@{=}[r]   & F_0(X_0) \ar[r]^{\sigma}_{\cong}   & F_1(X_1) .
    }
\end{equation*}
The commutativity of the left square is a general property of adjunctions (\cite{ML} IV.1 Thm.\ 1(ii)), the one of the right square is the naturality of the counit $\varepsilon_0$. If we combine this diagram with $F_0$ applied to the previous diagram we obtain the commutative diagram
\begin{equation*}
  \xymatrix{
    F_0 Q_0 V(X_0,X_1,\sigma) \ar[dddd]^{F_0(\gamma_{0,X_0} \circ Q_0(\pi_{0,\Theta}))} \ar[rr]^-{F_0 Q_0(\pi_{1,\Theta})} && F_0 Q_0 U_1(X_1) \ar[d]^{F_0 Q_0 U_1(\eta_{1,X_1})} \ar@{=}[r] & F_1 Q_1 U_1(X_1)  \ar[d]^{F_1 Q_1 U_1(\eta_{1,X_1})}  \\
          && F_0 Q_0 U_1 H_1 F_1(X_1)  \ar@{=}[r] &  F_1 Q_1 U_1 H_1 F_1(X_1)         \\
          && F_0 Q_0 U_0 H_0 F_1(X_1) \ar[u]_{F_0 Q_0 (\iota_{F_1(X_1)})}^{\cong} \ar[d]^{F_0(\gamma_{0,H_0 F_1(X_1)})}_{\cong} \ar@{=}[r] & F_1 Q_1 U_0 H_0 F_1(X_1) \ar[u]_{F_1 Q_1 (\iota_{F_1(X_1)})}^{\cong}  \\
          && F_0 H_0 F_1(X_1) \ar[d]^{\varepsilon_{0,F_1(X_1)}}_{\cong}  & \\
    F_0 (X_0)\ar[rr]^{\sigma} && F_1(X_1) & .  }
\end{equation*}
By the analog for $X_1$ of our earlier argument for $X_0$ the diagram
\begin{equation*}
  \xymatrix{
    F_1 Q_1 U_1(X_1) \ar[d]_{F_1 Q_1 U_1(\eta_{1,X_1})} \ar[rr]^-{F_1(\gamma_{1,X_1})}_-{\cong} && F_1(X_1) \ar[d]^{F_1(\eta_{1,X_1})} \ar@{=}[r] & F_1(X_1) \ar@{=}[d] \\
    F_1 Q_1 U_1 H_1 F_1(X_1) \ar[rr]^-{F_1(\gamma_{1,H_1 F_1(X_1)})}_-{\cong} && F_1 H_1 F_1(X_1) \ar[r]^-{\varepsilon_{1,F_1(X_1)}}_-{\cong} & F_1(X_1)
    }
\end{equation*}
is commutative. The combination of these last two diagrams results in the commutative diagram
\begin{equation*}
  \xymatrix{
    F_0 Q_0 V(X_0,X_1,\sigma) \ar[d]_{F_0(\gamma_{0,X_0} \circ Q_0(\pi_{0,\Theta}))} \ar[rrr]^-{F_1(\gamma_{1,X_1} \circ Q_1(\pi_{1,\Theta}))}   &&&  F_1(X_1) \\
    F_0(X_0)  \ar[rrr]^-{\sigma} &&& F_1(X_1) \ar[u]_{?}         }
\end{equation*}
where $?$ is the solid arrow composite isomorphism
\begin{equation*}
  \xymatrix{
    F_0 Q_0 U_0 H_0 F_1(X_1) \ar[d]_{F_0(\gamma_{0,H_0 F_1(X_1)})}^{\cong} \ar[rr]^{F_0 Q_0(\iota_{F_1(X_1)})}_{\cong} && F_0 Q_0 U_1 H_1 F_1(X_1) \ar@{=}[r] & F_1 Q_1 U_1 H_1 F_1(X_1) \ar[d]^{F_1(\gamma_{1,H_1 F_1(X_1)})}_{\cong} \\
    F_0 H_0 F_1(X_1) \ar[d]_{\varepsilon_{0,F_1(X_1)}}^{\cong}  &&& F_1 H_1 F_1(X_1) \ar[d]^{\varepsilon_{1,F_1(X_1)}}_{\cong} \\
    F_1(X_1) \ar[rrr]^{?} &&& F_1(X_1) .  }
\end{equation*}
This reduces us to showing that $?$, in fact, is the identity. By Lemma \ref{lem:adjoints}.ii the compositions of the perpendicular arrows are the counits $\tau_i$ of the adjunctions between $F_0 \circ Q_0 = F_1 \circ Q_1$ and $U_i \circ H_i$. Hence the above diagram can be rewritten as
\begin{equation*}
  \xymatrix{
    F_0 Q_0 U_0 H_0 F_1(X_1)  \ar[d]_{\tau_{0,F_1(X_1)}}^{\cong} \ar[rr]^{F_0 Q_0(\iota_{F_1(X_1)})}_{\cong} && F_0 Q_0 U_1 H_1 F_1(X_1) \ar[d]^{\tau_{1,F_1(X_1)}}_{\cong} \\
    F_1(X_1) \ar[rr]^{\id} && F_1(X_1)  . }
\end{equation*}
It is commutative by Lemma \ref{lem:adjoints}.i.

By the first part of the assertion the functor $R$ is full and is essentially surjective on objects. But by Lemma \ref{lem:Rfaithful} it is also faithful. Hence is must be an equivalence.
\end{proof}

Since we will not use it we leave it to the reader to verify that the natural transformation \eqref{f:RVunit} is the unit for the adjunction between $R$ and $V$.

The above result greatly generalizes \cite{Gab} Prop.\ 2 on p.\ 441.

\begin{proposition}\label{prop:finite-sheaf}
    For any finitely many stable localizing subcategories $\cL_1, \ldots , \cL_n$ of $\cC$ the natural maps
\begin{equation*}
\xymatrix@C=0.5cm{
  0 \ar[r] & Z \big(\cC /(\cL_1 \cap \ldots \cap \cL_n) \big) \ar[r]^{} & \prod_{i=1}^n Z(\cC/\cL_i) \ar@<1ex>[r]{} \ar@<-1ex>[r]{} & \prod_{i,j}Z\big(\cC / (\cL_i \vee \cL_j) \big) }
\end{equation*}
   form an exact sequence.
\end{proposition}
\begin{proof}
By induction w.r.t.\ $n$ we only need to establish the case $n = 2$. By \cite{Gab} p.\ 446 we always have the exact sequence
\begin{equation*}
\xymatrix{
  0 \ar[r] & Z \big((\cC/\cL_1) \underset{\cC / (\cL_1 \vee \cL_2)}{\prod}{} (\cC/\cL_2) \big) \ar[r]^{} & Z(\cC/\cL_1) \times Z(\cC/\cL_2) \ar@<1ex>[r]{} \ar@<-1ex>[r]{} & Z(\cC / (\cL_1 \vee \cL_2)) \ . }
\end{equation*}
But under the stability assumption we may, by Prop.\ \ref{prop:R-equiv}, identify the left hand term with $Z \big(\cC /(\cL_1 \cap \cL_2) \big)$.
\end{proof}

\section{The central sheaf}\label{sec:central-sheaf}

We will make precise in which way $\cL \mapsto Z(\cC/\cL)$ is a sheaf.

The set $\mathbf{L}(\cC)$ partially ordered by inclusion together with $\cap$ and $\vee$ is a lattice. By Prop.\ \ref{prop:vee-stable} the subset $\mathbf{L}^{st}(\cC) \subseteq \mathbf{L}(\cC)$ of all stable localizing subcategories is a sublattice.

\begin{proposition}\label{prop:distributive}
   The lattice $\mathbf{L}(\cC)$ is distributive.
\end{proposition}
\begin{proof}
The proof given in \cite{Go2} Prop.\ 29.1 in case $\cC = \Mod(R)$ for any ring $R$ works in general. For the convenience of the reader we provide one more detail. For avoiding confusion we first point out that Golan defines torsion theories always to be hereditary and therefore to correspond to localizing subcategories. The proof in loc.\ cit. reduces the assertion to the following claim: Suppose that an object $X$ in $\cC$ is $\cL_i$-torsion-free for $i = 1, 2$; then it is also $\cL_1 \vee \cL_2$-torsion-free. To see this let $\cL(X)$ denoted the localizing subcategory cogenerated by the injective hull $E(X)$. Since $X$ is $\cL_i$-torsion-free one has $\cL_i \subseteq \cL(X)$. It follows that $\cL_1 \vee \cL_2 \subseteq \cL(X)$. But by definition $E(X)$ is $\cL(X)$-torsion-free and hence is a fortiori $\cL_1 \vee \cL_2$-torsion-free. Therefore $X$ must be $\cL_1 \vee \cL_2$-torsion-free.
\end{proof}

We now view $\mathbf{L}^{st}(\cC)$ as a category with objects being the elements of $\mathbf{L}^{st}(\cC)$ and where the morphisms $\cL_2 \rightarrow \cL_1$ are given by the inclusions $\cL_1 \subseteq \cL_2$. For $\cL$ and $\cL_1, \ldots, \cL_n$ in $\mathbf{L}^{st}(\cC)$ we call $\{\cL_i\}_{1 \leq i \leq n}$ a covering of $\cL$ if $\cL = \bigcap_i \cL_i$.

\begin{lemma}\label{lem:site}
   With the above notion of coverings the category $\mathbf{L}^{st}(\cC)$ is a Grothendieck site.
\end{lemma}
\begin{proof}
Since any diagram in $\mathbf{L}^{st}(\cC)$ of the form
\begin{equation*}
  \xymatrix{
    \cL_1 \vee \cL_2 \ar[d] \ar[r] & \cL_2 \ar[d] \\
    \cL_1 \ar[r] & \cL   }
\end{equation*}
is a fiber product diagram the category $\mathbf{L}^{st}(\cC)$ has fiber products. Obviously coverings compose into coverings. That any base change of a covering again is a covering is immediate from the distributivity in Prop.\ \ref{prop:distributive}.
\end{proof}

Obviously $\cL \mapsto Z(\cC/\cL)$ is a presheaf on the site $\mathbf{L}^{st}(\cC)$.

\begin{theorem}\label{thm:central-sheaf}
  $Z_\cC(\cL) := Z(\cC/\cL)$ is a sheaf on $\mathbf{L}^{st}(\cC)$.
\end{theorem}
\begin{proof}
This is now immediate from Prop.\ \ref{prop:finite-sheaf}.
\end{proof}

\section{The locally noetherian case}\label{sec:loc-noetherian}

The injective spectrum $\mathbf{Sp}(\cD)$ of a Grothendieck category is defined to be the collection of isomorphism classes of indecomposable injective objects of $\cD$. It is a set (cf.\ \cite{Pop} p.\ 331).

Throughout this section we assume that our Grothendieck category $\cC$ is locally noetherian. We will construct an alternative version of the central sheaf $Z_\cC$ which will be a sheaf on the topological space $\mathbf{Sp}(\cC)$ equipped with the so-called stable topology.

\begin{remark}\label{rem:loc-noeth}
  Let $\cL$ be any localizing subcategory of $\cC$; then $\cL$ and $\cC/\cL$ are locally noetherian, and the section functor $\cC/\cL \rightarrow \cC$ commutes with inductive limits.
\end{remark}
\begin{proof}
\cite{Gab} Cor.\ 1 on p.\ 379.
\end{proof}

For any localizing subcategory $\cL$ of a Grothendieck category $\cD$ one defines the subset
\begin{equation*}
  A(\cL) := \{ [E] \in \mathbf{Sp}(\cD) : \Hom_\cC(V,E) = 0\ \text{for any $V \in \ob(\cL)$}\}
\end{equation*}
of $\mathbf{Sp}(\cD)$. In the case of $\cC$ these subsets $A(\cL)$ form the closed subsets of a topology on $\mathbf{Sp}(\cC)$ which is called the Ziegler topology (cf.\ \cite{Her} Thm.\ 3.4). In fact, by \cite{Her} Theorems 2.8 and 3.8, the map
\begin{align}\label{f:loc-classific}
  \text{collection of all localizing} & \xrightarrow{\;\simeq\;} \text{set of all Ziegler-closed} \\
  \text{subcategories of $\cC$} &                    \text{\qquad\ subsets of $\mathbf{Sp}(\cC)$}    \nonumber \\
                          \cL & \longmapsto A(\cL)       \nonumber
\end{align}
is an inclusion reversing bijection. This means that the Ziegler-closed subsets of $\mathbf{Sp}(\cC)$ classify the localizing subcategories of $\cC$. It implies that $\cL$ can be reconstructed from $A(\cL)$ by
\begin{equation*}
  \ob(\cL) = \{ V \in \ob(\cC) : \Hom_\cC(V,E) = 0 \ \text{for all $[E] \in A(\cL)$}\}.
\end{equation*}
It also implies that
\begin{equation*}
  A(\cL_1 \cap \cL_2) = A(\cL_1) \cup A(\cL_2)  \qquad\text{and}\qquad A(\cL_1 \vee \cL_2) = A(\cL_1) \cap A(\cL_2) \ .
\end{equation*}

\begin{lemma}\label{lem:stable-1}
   For any localizing subcategory $\cL$ of $\cC$ the following are equivalent:
\begin{itemize}
  \item[(a)] $\cL$ is stable;
  \item[(b)] any indecomposable injective object of $\cC$ either lies in $\cL$ or has no non-zero subobject lying in $\cL$.
\end{itemize}
\end{lemma}
\begin{proof}
We argue similarly as in \cite{Gol} Prop.\ 11.3.

$(a) \Longrightarrow (b)$: Let $E$ be an indecomposable injective object in $\cC$. Suppose that $t_\cL(E) \neq 0$. Then $E$ is an injective hull of $t_\cL(E)$ and hence, by stability, is contained in $\cL$.

$(b) \Longrightarrow (a)$: Since $\cC$ is locally noetherian we may write an injective hull $E(V)$ of an object $V$ lying in $\cL$ as a direct sum $E(V) = \oplus_{i \in I} E_i$ of indecomposable injective objects $E_i$. Since $E_i \cap V \neq 0$ lies in $\cL$ for any $i \in I$ we see that all $E_i$ and hence $E(V)$ lie in $\cL$.
\end{proof}

\begin{corollary}\label{cor:stable}
   Let $\cL$ be a stable localizing subcategory of $\cC$; then
\begin{equation*}
  A(\cL) = \{[E] \in \mathbf{Sp}(\cC) : E \not\in \ob(\cL)\}  \ .
\end{equation*}
\end{corollary}

The following is a straightforward generalization of \cite{Lou} Prop.\ 4.

\begin{lemma}\label{lem:stable-2}
   For a subset $A \subseteq \mathbf{Sp}(\cC)$ the following are equivalent:
\begin{itemize}
  \item[(a)] $A = A(\cL)$ for a stable localizing subcategory $\cL$ of $\cC$;
  \item[(b)] if $[E] \in \mathbf{Sp}(\cC)$ satisfies $\Hom_\cC(E,E') \neq 0$ for some $[E'] \in A$ then $[E] \in A$.
\end{itemize}
\end{lemma}
\begin{proof}
$(a) \Longrightarrow (b)$: Let $[E'] \in A(\cL)$ such that $\Hom_\cC(E,E') \neq 0$. Then $E$ does not lie in $\cL$. Since $\cL$ is stable Lemma \ref{lem:stable-1} applies and tells us that $E$ does not have any non-zero subobject lying in $\cL$. Hence $[E] \in A(\cL) = A$.

$(b) \Longrightarrow (a)$: Let $\cL$ be the localizing subcategory of $\cC$ cogenerated by the $E'$ for $[E'] \in A$. This means that $\cL$ is the full subcategory of those objects $V$ in $\cC$ which satisfy $\Hom_\cC(V,E') = 0$ for any $[E'] \in A$. It is immediate that $A \subseteq A(\cL)$. Consider any $[E] \in A(\cL)$. Then $E$ cannot lie in $\cL$. Hence there must exist an $[E'] \in A$ such that $\Hom_\cC(E,E') \neq 0$. It follows from (b) that $[E] \in A$. This shows that $A = A(\cL)$. To establish that $\cL$ is stable we use Lemma \ref{lem:stable-1}. We have just seen that the $E$ which do not lie in $\cL$ must have $[E] \in A(\cL)$. By the very definition of $A(\cL)$ such $E$ do not have a non-zero subobject lying in $\cL$.
\end{proof}

A subset $A \subseteq \mathbf{Sp}(\cC)$ will be called stable if it is of the form $A = A(\cL)$ for some stable localizing subcategory $\cL$ of $\cC$. It is clear, for example from Lemma \ref{lem:stable-2}, that arbitrary intersections and unions of stable subsets are stable again. Therefore the stable subsets are the open subsets for a topology which we call the stable topology of $\mathbf{Sp}(\cC)$.

For any stable subset $A \subseteq \mathbf{Sp}(\cC)$ let $\cL_A$ denote the stable localizing subcategory of $\cC$ such that $A(\cL_A) = A$. The map $A \mapsto \cL_A$ is inclusion reversing. Hence, for $A_1 \subseteq A_2$, we have a quotient functor $\cC/\cL_{A_2} \rightarrow \cC/\cL_{A_1}$. This means that
\begin{equation}\label{f:presheaf}
  A \ \text{stable} \longmapsto \cZ_\cC(A) := Z(\cC/\cL_A)
\end{equation}
is a presheaf of commutative rings on $\mathbf{Sp}(\cC)$ for the stable topology. Our goal in this section is to prove the following theorem.

\begin{theorem}\label{thm:sheaf}
  $\cZ_\cC$ is a sheaf on $\mathbf{Sp}(\cC)$ for the stable topology.
\end{theorem}

We first need several preparations.

\subsection{Center via injective cogenerators}

We will write an element in $Z(\cC)$ often as $z = (z_M)_M$ with $z_M = \ev_M(z)$ denoting the endomorphism of the object $M$ in $\cC$ defined by $z$ and being called the evaluation of $Z$ in $M$.

Recall (cf.\ \cite{Ste} \S IV.6) that an injective object $E$ in $\cC$ is a cogenerator of $\cC$ if and only if any non-zero object $M$ in $\cC$ has a non-zero homomorphism $M \rightarrow E$. Moreover, since $\cC$ has arbitrary products (\cite{Ste} Cor.\ X.4.4), there is then a monomorphism $M \hookrightarrow \prod_{i \in I} E$ for some index set $I$.

For any $x \in \mathbf{Sp}(\cC)$ we fix a representative $E_x$ in the isomorphism class $x = [E_x]$. We introduce the following objects
\begin{equation*}
  E_\cC^\oplus := \bigoplus_{x \in \mathbf{Sp}(\cC)} E_x  \qquad\text{and}\qquad  E_\cC^\pi := \prod_{x \in \mathbf{Sp}(\cC)} E_x \ .
\end{equation*}
Both are injective objects, the former since $\cC$ is locally noetherian (\cite{Gab} Prop.\ 6 on p.\ 387) and the latter for formal reasons.

\begin{lemma}
  $E_\cC^\oplus$ and $E_\cC^\pi$ are cogenerators of the category $\cC$.
\end{lemma}
\begin{proof}
Consider any object $0 \neq M \in \ob(\cC)$. By Matlis' theorem (\cite{Gab} Thm.\ 2 on p.\ 388) its injective hull $E(M)$ decomposes as a direct sum of indecomposable injective objects, i.e., we have $E(M) \cong \oplus_i E_i \hookrightarrow \prod_i E_i$ with $x_i := [E_i] \in \mathbf{Sp}(\cC)$. Hence we find an $i_0$ such that the composite homomorphism $M \xrightarrow{\subseteq} E(M) \hookrightarrow \oplus_i E_i \xrightarrow{\pr_{i_0}} E_{i_0} \cong E_{x_{i_0}} \hookrightarrow E_\cC^\oplus \hookrightarrow E_\cC^\pi$ is non-zero.
\end{proof}

\begin{definition}
  An injective cogenerator $E$ of $\cC$ is called a good cogenerator if it has a subobject (and hence a direct factor) which is isomorphic to $E_\cC^\oplus$.
\end{definition}

Obviously $E_\cC^\oplus$ and $E_\cC^\pi$ are good cogenerators. The reason for this definition is the following fact.

\begin{lemma}\label{lem:good-co}
  Let $E$ be a good cogenerator of $\cC$; then, for any object $M$ in $\cC$, there is an index set $I$ and a monomorphism $M \hookrightarrow \oplus_{i \in I} E$.
\end{lemma}
\begin{proof}
By Matlis' theorem an injective hull $E(M)$ of $M$ is of the form $E(M) \cong \oplus_{x \in \mathbf{Sp}(\cC)} \oplus_{i \in I_x} E_x$ for appropriate index sets $I_x$. By embedding all $I_x$ into a common index set $I$ the right hand side is contained in $\oplus_{i \in I} E_\cC^\oplus$. We obtain the sequence of monomorphisms $M \hookrightarrow E(M) \hookrightarrow \oplus_{i \in I} E_\cC^\oplus \hookrightarrow \oplus_{i \in I} E$, where the last one exists by our assumption that $E$ is good.
\end{proof}

Our goal in this section is to prove the following result.

\begin{theorem}\label{thm:BerCent}
  The evaluation map $\ev_E : Z(\cC) \xrightarrow{\cong} Z(\End_\cC(E))$, for any good cogenerator $E$ of $\cC$, is an isomorphism.
\end{theorem}

For the proof we will construct a map in the opposite direction, using the definition of $Z(\cC)$. We begin with the following observation. Let $\{X_i: i \in I\}$ and $\{Y_j: j \in J\}$ be two collections of objects in $\cC$. Set $X := \oplus_{i \in I} X_i$ and $Y := \oplus_{j \in J} Y_j$. Let $\iota_i : X_i \to X$ and $\pi_j : Y \to Y_j$ be the canonical inclusions and projections. Then there is the natural composed map
\begin{equation}\label{f:mu}
   \mu : \Hom_\cC(X,Y) = \prod_i \Hom_\cC(X_i, \oplus_j Y_j) \hookrightarrow \prod_i \Hom_\cC(X_i, \prod_j Y_j) = \prod_{i,j} \Hom_\cC(X_i,Y_j)
\end{equation}
that sends $\varphi \in \Hom_\cC(X,Y)$ to its \emph{matrix} $\mu(\varphi) := (\varphi_{ij})$, where $\varphi_{ij} := \pi_j  \varphi \iota_i$ for all $i,j$. It is visibly injective. \footnote{In Grothendieck categories we have the ``usual'' map $\oplus_j Y_j \rightarrow \prod_j Y_j$ and it is a monomorphism; see Schubert H., Kategorien I, Satz 14.6.8.}

Suppose now that $E$ is a good cogenerator, and let $z \in Z(\End_\cC(E))$. To define an element of $Z(\cC)$ that corresponds to this $z$, we must construct elements $z_M \in \End_\cC(M)$ for all objects $M$ in $\cC$ that commute with all homomorphisms in $\cC$. Obviously, if $M = \oplus_{i \in I} E$ we define $z_M := \oplus_i \, z$. To simplify notation we write in the following $E^{(I)} := \oplus_{i \in I} E$ for any index set $I$.

\begin{lemma}\label{lem:HomXfXg}
Let $I$ and $J$ be two index sets. Then for all $\varphi \in \Hom_\cC(E^{(I)},E^{(J)})$, the following square is commutative:
\begin{equation*}
  \xymatrix{
  E^{(I)} \ar[d]_{z_{E^{(I)}}}\ar[rr]^\varphi && E^{(J)} \ar[d]^{z_{E^{(J)}}} \\
  E^{(I)} \ar[rr]_\varphi  && E^{(J)}
  }
\end{equation*}
\end{lemma}
\begin{proof}
For $i \in I, j\in J$ and let $\iota_i : E  \hookrightarrow E^{(I)}$ and $\pi_j : E^{(J)} \twoheadrightarrow E$ be the inclusion into the $i$th summand and the projection onto the $j$th summand, respectively. By the definitions of $z_{E^{(I)}}$ and $z_{E^{(J)}}$, we have $\pi_j z_{E^{(J)}}  = z \pi_j$ and $\iota_i z = z_{E^{(I)}} \iota_i$. Let $v := \pi_j \varphi \iota_i \in \End_\cC(E)$; then $zv = vz$ because $z \in Z(\End_\cC(E))$. Therefore
\begin{equation*}
  \pi_j z_{E^{(J)}} \varphi \iota_i = z \pi_j \varphi \iota_i = z v = v z
   =  \pi_j \varphi \iota_i z = \pi_j \varphi z_{E^{(I)}} \iota_i \ .
\end{equation*}
Hence $z_{E^{(J)}} \varphi = \varphi z_{E^{(I)}}$ by the injectivity of $\mu$.
\end{proof}

Now, given an arbitrary object $M$ in $\cC$, we can find, by our definition of a good cogenerator, an exact sequence of the form
\begin{equation}\label{eq:copres}
  0 \to M \xrightarrow{\;\eta\;} E^{(I)} \to E^{(J)} \ ,
\end{equation}
and we can define $z_M : M \to M$ to be the unique homomorphism in $\cC$ which makes the following diagram commutative:
\begin{equation*}
  \xymatrix{
  0 \ar[r] & M \ar[r] \ar[d]_{z_M} & E^{(I)} \ar[r]\ar[d]_{z_{E^{(I)}}} & E^{(J)} \ar[d]_{z_{E^{(J)}}} \\
  0 \ar[r] & M \ar[r] & E^{(I)} \ar[r] & E^{(J)}
  }
\end{equation*}

\begin{lemma}
  $z_M$  does not depend on the choice of the exact sequence \eqref{eq:copres}.
\end{lemma}
\begin{proof}
Suppose that $0 \to M \xrightarrow{\;\eta'\;} E^{(I')} \to E^{(J')}$ is another exact sequence which gives rise to $z'_M : M \to M$; then since $E^{(I')}$ is injective, we can find a homomorphism $\sigma : E^{(I)} \to E^{(I')}$ which makes
\begin{equation*}
  \xymatrix{
  0 \ar[r] & M \ar[r]^\eta \ar[dr]_{\eta'} & E^{(I)} \ar@{..>}[d]^\sigma \\
  & & E^{(I')}
  }
\end{equation*}
commutative. Now consider the following diagram:
\begin{equation*}
  \xymatrix{
   E^{(I)} \ar[rrrr]^{z_{E^{(I)}}} \ar[dddd]_\sigma &&&& E^{(I)} \ar[dddd]^\sigma \\
    & M \ar[ul]^{\eta} \ar@{=}[dd] \ar[rr]^{z_M} && M \ar[ur]_\eta\ar@{=}[dd] & \\
     & & 0\ar[ur]\ar[ul]\ar[dr]\ar[dl] & & \\ & M \ar[rr]_{z'_M} \ar[dl]_{\eta'}& & M \ar[dr]^{\eta'}& \\
      E^{(I')} \ar[rrrr]_{z_{E^{(I')}}} &&&& E^{(I')} .
       }
\end{equation*}
The outer square is commutative by Lemma \ref{lem:HomXfXg}. The two trapezia on the sides commute by definition of $\sigma$. The two trapezia on the top and bottom commute by the definition of $z_M$ and $z'_M$, respectively. Chasing this diagram, we find that
\begin{equation*}
  \eta' z_M = \sigma \eta z_M = \sigma z_{E^{(I)}} \eta = z_{E^{(I')}} \sigma \eta = z_{E^{(I')}} \eta' = \eta' z'_M \ .
\end{equation*}
Since $\eta'$ is a monomorphism, we conclude that $z_M = z'_M$ as required.
\end{proof}

\begin{lemma}\label{lem:CentInCat}
Let $\theta : M \to M'$ be a homomorphism in $\cC$. Then $z_{M'} \theta = \theta z_M$.
\end{lemma}
\begin{proof}
Choose exact sequences $0 \to M \xrightarrow{\;\eta\;} E^{(I)} \to E^{(J)}$ and $0 \to M' \xrightarrow{\;\eta'\;} E^{(I')} \to E^{(J')}$. Using the injectivity of $E^{(I')}$ together with Lemma \ref{lem:HomXfXg}, we can find a similar commutative diagram
\begin{equation*}
  \xymatrix{
    E^{(I)} \ar[rrrr]^{z_{E^{(I)}}} \ar[dddd]_\sigma &&&& E^{(I)} \ar[dddd]^\sigma \\
    & M \ar[ul]^{\eta} \ar[dd]_\theta \ar[rr]^{z_M} && M \ar[ur]_\eta\ar[dd]^\theta & \\
    & & 0\ar[ur]\ar[ul]\ar[dr]\ar[dl] & & \\ & M' \ar[rr]_{z_{M'}} \ar[dl]_{\eta'}& & M' \ar[dr]^{\eta'}& \\
    E^{(I')} \ar[rrrr]_{z_{E^{(I')}}} &&&& E^{(I')}.
     }
\end{equation*}
Chasing this diagram, we similarly find that
\begin{equation*}
  \eta' \theta z_M = \sigma \eta z_M = \sigma z_{E^{(I)}} \eta = z_{E^{(I')}} \sigma \eta = z_{E^{(I')}} \eta'\theta  = \eta' z_{M'} \theta \ .
\end{equation*}
Because $\eta'$ is a monomorphism, $\theta z_M = z_{M'} \theta$ as required.
\end{proof}

\begin{proof}[Proof of Theorem \ref{thm:BerCent}]
Let $z \in Z(\End_\cC(E))$ and define $\psi(z) := (z_M)_{M \in \ob(\cC)}$ as constructed above. Then Lemma \ref{lem:CentInCat} shows that $\psi(z) \in Z(\cC)$.  By construction we have $\ev_E (\psi(z)) = z$. Therefore $\ev_E$ is surjective.

For the injectivity let now $z \in Z(\cC)$ be such that $z_E = 0$. For any homomorphism $f : U \rightarrow E$ we have $f \circ z_U = z_E \circ f = 0$. But the cogenerator property means that the functor $\Hom_\cC(-,E)$ is faithful. It follows that $z_U = 0$.
\end{proof}

\subsection{The sheaf property}

We will proceed by comparing the presheaf $\cZ_\cC$ with another presheaf defined as follows. Recall that we have fixed representatives of the isomorphism classes $x = [E_x] \in \mathbf{Sp}(\cC)$. For every subset $A$ of $\mathbf{Sp}(\cC)$, we define
\begin{equation*}
  \cF(A) := \{ z = (z_x)_x \in \prod_{x \in A} Z(\End_\cC(E_x)) : z_y v = v z_x \ \text{for all $x, y \in A$ and $v \in \Hom_\cC(E_x, E_y)$}\} .
\end{equation*}
Evidently, $\cF$ forms a presheaf on $\mathbf{Sp}(\cC)$ for the discrete topology.

\begin{lemma}\label{lem:Fsheaf}
$\cF$ is a sheaf for the stable topology of $\mathbf{Sp}(\cC)$.
\end{lemma}
\begin{proof}
Suppose that $\{A(i)\}_{i \in I}$ is a covering of some stable subset $A$ of $\mathbf{Sp}(\cC)$ by stable subsets $A(i)$, and let $z(i) \in \cF(A(i))$ be given for each $i \in I$, such that $z(i)_{|A(i) \cap A(j)} = z(j)_{|A(i) \cap A(j)}$ for each $i,j \in I$. Then for each $x \in A$, we may unambiguously define $z_x := z(i)_x \in Z(\End_\cC(E_x))$ for any index $i \in I$ such that $x \in A(i)$; this gives us a vector $z \in \prod_{x \in A} Z(\End_\cC(E_x))$ such that $z_{|A(i)} = z(i)$ for all $i \in I$. We must show that $z \in \cF(A)$. To this end, let $x, y \in A$ and let $v \in \Hom_\cC(E_x, E_y)$; we must show that $z_y v = v z_x$.  If $v = 0$ there is nothing to do, so we may assume that $v\neq 0$. Now, $y \in A(i)$ for some $i \in I$; since $A(i)$ is stable, also $x \in A(i)$ by Lemma \ref{lem:stable-2}. But then $z_y v = v z_x$ since $x,y \in A(i)$ and $z_{|A(i)} = z(i) \in \cF(A(i))$.
\end{proof}

For any subset $A \subseteq \mathbf{Sp}(\cC)$ let
\begin{equation*}
  \cL_A := \ \text{localizing subcategory of $\cC$ cogenerated by the injective object $E_A := \oplus_{x \in A} E_x$}.
\end{equation*}
This means (cf.\ \cite{Ste} proof of Prop.\ VI.3.7) that
\begin{equation*}
  \cL_A \ \text{is the full subcategory of all $M \in \ob(\cC)$ such that $\Hom_\cC(M,E_A) = 0$}.
\end{equation*}

\begin{remark}\label{rem:L}
   Suppose that $A = A(\cL)$ for some localizing subcategory $\cL$ of $\cC$; then $\cL_A = \cL$. For this we recall from the beginning of this section that
\begin{equation*}
  A(\cL) = \ \{x \in \mathbf{Sp}(\cC) : \Hom_\cC(M,E_x) = 0 \ \text{for any $M \in \ob(\cL)$}\}
\end{equation*}
and
\begin{align*}
  \ob(\cL) & = \{ M \in \ob(\cC) : \Hom_\cC(M,E_x) = 0 \ \text{for all $x \in A(\cL)$}\}   \\
         & = \{ M \in \ob(\cC) : \Hom_\cC(M, \prod_{x \in A(\cL)} E_x) = 0 \} .
\end{align*}
Obviously, if $\Hom_\cC(M, \prod_{x \in A(\cL)} E_x) = 0$ then $\Hom_\cC(M, \oplus_{x \in A(\cL)} E_x) = 0$ as well. Vice versa, suppose that $\Hom_\cC(M, \oplus_{x \in A(\cL)} E_x) = 0$. Assuming that $\Hom_\cC(M, \prod_{x \in A(\cL)} E_x) \neq 0$ we find an $y \in A(\cL)$ such that $\Hom_\cC(M, E_y) \neq 0$, which is a contradiction. We conclude that
\begin{equation*}
   \ob(\cL) = \{ M \in \ob(\cC) : \Hom_\cC(M, \oplus_{x \in A(\cL)} E_x) = 0 \}  = \ob(\cL_{A(\cL)}) \ .
\end{equation*}
\end{remark}

Also recall from Lemma \ref{lem:stable-2} that $\cL_A$ is stable if and only if $A$ is stable.

For any localizing subcategory $\cL$ of $\cC$ we let $q_\cL : \cC \rightarrow \cC/\cL$ denote the quotient functor and $s_\cL$ its right adjoint section functor. Recall that $s_\cL$ is fully faithful.

\begin{lemma}\label{lem:Lclosed}
Let $\cL$ be a stable localizing subcategory of $\cC$. Then the quotient functor $q_{\cL}$ restricts to an equivalence of categories between the full subcategory of all injective objects of $\cC$ which have no non-zero subobject contained in $\cL$ and the full subcategory of all injective objects of $\cC/\cL$.
\end{lemma}
\begin{proof}
By \cite{Gab} Lemma 1 on p.\ 370 an injective object of $\cC$ has no non-zero subobject contained in $\cL$ if and only if it is $\cL$-closed. We therefore consider the full subcategory $\cD$ of all $\cL$-closed objects of $\cC$. Then the restriction of $q_{\cL}$ to $\cD$ is an equivalence of categories with quasi-inverse $s_{\cL}$, by \cite{Gab} Prop.\ 3(a) and Cor.\ on p. 371. So it is enough to show that $q_\cL$ sends the $\cL$-closed injectives in $\cC$ to injectives in $\cC/\cL$, and that $s_{\cL}$ sends injectives in $\cC/\cL$ to the $\cL$-closed injectives in $\cC$.

Let $I$ be an $\cL$-closed injective object in $\cC$. Then because $\cL$ is stable, $q_{\cL} (I)$ is an injective object in $\cC/\cL$ by \cite{Gab} Cor.\ 3 on p.\ 375. Conversely, let $I$ be an injective object in $\cC/\cL$; then $s_{\cL} (I)$ is an injective object in $\cC$ since $q_{\cL}$ is an exact left adjoint to $s_{\cL}$, and furthermore $s_\cL (I)$ is $\cL$-closed by \cite{Gab} Lemma 2 on p.\ 371.
\end{proof}

The injective objects $E_x$, for $x \in A(\cL)$, by definition, and also the injective object $E_{A(\cL)}$, by Remark \ref{rem:L}, have no non-zero subobject contained in $\cL$. We now assume that $\cL$ is stable. It follows from \cite{Gab} last paragraph on p.\ 383 (using the stability of $\cL$) that the $[q_\cL (E_x)]$, for $x \in A(\cL)$, are precisely the elements of $\mathbf{Sp}(\cC/\cL)$, so that we choose these $q_\cL (E_x)$ as representatives. The functor $q_{\cL}$ is a left adjoint and therefore commutes with arbitrary direct sums; hence $q_{\cL}(E_{A(\cL)}) = \oplus_{x \in A(\cL)} \, q_{\cL} (E_x) = E_{\cC/\cL}^\oplus$. The Lemma \ref{lem:Lclosed} then implies that
\begin{equation*}
  E_{A(\cL)} \cong s_\cL(q_{\cL}(E_{A(\cL)})) = s_\cL(E_{\cC/\cL}^\oplus) \ .
\end{equation*}
But the functor $s_\cL$ is fully faithful. Therefore we obtain the following consequence.

\begin{corollary}\label{cor:Qiso}
Let $\cL$ be a stable localizing subcategory of $\cC$.  Then $q_{\cL}(E_{A(\cL)}) = E_{\cC/\cL}^\oplus$, and the functor $q_\cL$ induces an isomorphism of rings
\begin{equation*}
  \End_{\cC}( E_{A(\cL)}) \xrightarrow{\;\cong\;} \End_{\cC/\cL} ( E_{\cC/\cL}^\oplus) \ .
\end{equation*}
\end{corollary}

Note that for the same reason the functor $q_\cL$ also induces isomorphisms of rings
\begin{equation}\label{f:Qiso}
  \End_{\cC}( E_x) \xrightarrow{\;\cong\;} \End_{\cC/\cL} (q_\cL (E_x))  \qquad\text{for any $x \in A(\cL)$}.
\end{equation}

At this point we need a generalization of the well-known calculation of the center of a matrix ring $M_n(R)$ over some associative ring $R$: central elements in $M_n(R)$ are necessarily scalar matrices with entries in $Z(R)$.

\begin{lemma}\label{lem:GenCentre}
Let  $\{X_i: i \in I\}$ be a collection of objects in $\cC$, and let $X := \oplus_{i \in I} X_i$. Then
\begin{multline*}
  Z(\End_\cC(X)) =  \\
   \{ (z_i)_i \in \prod_{i \in I} Z(\End_\cC(X_i)) : z_j v = v z_i \ \text{for all $i, j \in I$ and $v \in \Hom_\cC(X_i,X_j)$} \}.
\end{multline*}
\end{lemma}
\begin{proof}
Recall the injective map $\mu$ defined in \eqref{f:mu}. Suppose that $\alpha \in Z(\End_\cC(X))$. Take some $v \in \Hom_\cC(X_i,X_j)$ and take any $k, \ell \in I$. Then $\pi_k \alpha (\iota_j v \pi_i) \iota_\ell = \pi_k (\iota_j v \pi_i) \alpha \iota_\ell$; since $\pi_i \iota_\ell = \delta_{i\ell} \id_{X_\ell}$ we deduce that
\begin{equation}\label{ijkl}
   \delta_{i\ell} \alpha_{jk } v = \delta_{kj} v \alpha_{\ell i} \quad \mbox{for all}\quad i,j,k,\ell \in I, v \in \Hom_\cC(X_i,X_j) \ .
   \end{equation}
Suppose that $j \neq k$. Take $\ell := i := j$ and $v := \id_{X_j}$ in (\ref{ijkl}) to deduce that $\alpha_{jk} = 0$ whenever $j\neq k$; thus the matrix $\mu(\alpha)$ of any $\alpha \in Z(\End_\cC(X))$ is diagonal. Fixing $i,j \in I$ and taking $\ell :=i$ and $k :=j$ in (\ref{ijkl}) shows that $\alpha_{jj} v = v \alpha_{ii}$ holds for all $v \in \Hom_\cC(X_i,X_j)$. Taking $i = j$ in this last equation shows that $\alpha_{ii} \in Z(\End_\cC(X_i))$ for all $i \in I$. This shows that the map
\begin{equation*}
  Z(\End_\cC(X)) \rightarrow \{ (z_i)_i \in \prod_{i \in I} Z(\End_\cC(X_i)) : z_j v = v z_i \ \text{for all $v \in \Hom_\cC(X_i,X_j)$} \}
\end{equation*}
that sends $\alpha \in Z(\End_\cC(X))$ to the vector $(\alpha_{ii})$ is well-defined; it is furthermore injective, because $\mu(\alpha)$ is diagonal and because $\mu$ is injective.

To show that the map is surjective, take any $(z_i)_i \in \prod_{i \in I} Z(\End_\cC(X_i))$ satisfying the given condition and define $\alpha := \oplus_i z_i$, an element of $\End_\cC(X)$. We must show that $\alpha$ is central. Note that $\pi_j \alpha = z_j \pi_j$ and $\alpha \iota_i = \iota_i z_i$ for all $i,j \in I$, by the definition of $\alpha$. Now if $\beta \in \End_\cC(X)$ and $i,j \in I$, then applying the condition on the $z_i$'s with $v := \pi_j \beta \iota_i \in \Hom_\cC(X_i,X_j)$, we have
\begin{equation*}
  \pi_j \alpha \beta \iota_i = z_j \pi_j \beta \iota_i = z_j v = v z_i = \pi_j \beta\iota_i z_i = \pi_j \beta\alpha\iota_i \ .
\end{equation*}
Hence $\mu(\alpha \beta) = \mu(\beta\alpha)$ for all $\beta \in \End_\cC(X)$. Since $\mu$ is injective, we conclude that $\alpha$ is central in $\End_\cC(X)$ as claimed.
\end{proof}

We let $\alpha_A : Z(\End_\cC(E_A)) \xrightarrow{\cong} \cF(A)$ be the isomorphism given by the above Lemma \ref{lem:GenCentre}.

\begin{proposition}\label{prop:FtoZmap}
Let $A \subseteq B$ be two stable subsets of $\mathbf{Sp}(\cC)$. Then there is a commutative diagram of commutative rings
\begin{equation}\label{eq:FtoZ}
\xymatrix{
\cF(B) \ar[d]_{\res} & Z( \End_{\cC} (E_B) ) \ar[l]_-{\alpha_B} \ar[r]^-{q_{\cL_B}} \ar[d]_\gamma & Z(\End_{\cC/\cL_B} (E_{\cC/\cL_B}^\oplus)) \ar[d]_{\overline{Q}}&& \cZ_\cC(B) \ar[ll]_-{\ev_{E_{\cC/\cL_B}^\oplus}} \ar[d]_{\res} \\
\cF(A) & Z( \End_{\cC} (E_A) ) \ar[l]^-{\alpha_A}  \ar[r]_-{q_{\cL_A}}& Z(\End_{\cC/\cL_A} (E_{\cC/\cL_A}^\oplus)) && \cZ_\cC(A)\ar[ll]^-{\ev_{E_{\cC/\cL_A}^\oplus}}.
}
\end{equation}
\end{proposition}
\begin{proof}
Note that $B = A(\cL_B)$ and $A = A(\cL_A)$ by Remark \ref{rem:L}. We let $\iota : E_A \hookrightarrow E_B$ and $\pi : E_B \twoheadrightarrow E_B$ be the natural inclusion and projection maps, and define $\gamma(\varphi) := \pi \circ \varphi \circ \iota$. Then the first square commutes by the definitions of $\alpha_B$ and $\alpha_A$.  Let $\overline{Q} : \cC/\cL_B \to \cC/\cL_A$ be the quotient functor, so that $\overline{Q} \circ q_{\cL_B} = q_{\cL_A}$. Then
\begin{equation*}
  \overline{Q} (E_{\cC/\cL_B}^\oplus) = \overline{Q} q_{\cL_B} (E_B) = \bigoplus_{x \in B} q_{\cL_A} (E_x) = \bigoplus\limits_{y \in A} q_{\cL_A} (E_y) = q_{\cL_A} (E_A) = E_{\cC/\cL_A}^\oplus \ .
\end{equation*}
Here the outer equalities come from Cor.\ \ref{cor:Qiso}. The middle equality holds true because $q_{\cL_A} (E_x) = 0$ for any $x \in B \setminus A$: such an $x$ does not lie in $A = A(\cL_A)$ as noted at the beginning of the proof, so the indecomposable injective $E_x$ lies in the stable subcategory $\cL_A$ by Lemma \ref{lem:stable-1} and hence $q_{\cL_A} (E_x) = 0$. We now see that the middle square commutes. Finally, the square on the right commutes because $\overline{Q} (E_{\cC/\cL_B}^\oplus) = E_{\cC/\cL_A}^\oplus$ as we saw above.
\end{proof}

\begin{proof}[Proof of Theorem \ref{thm:sheaf}]
Let $A$ be a stable subset of $\mathbf{Sp}(\cC)$. By Cor.\ \ref{cor:Qiso}, the morphism $q_{\cL_A}$ appearing in the diagram $(\ref{eq:FtoZ})$ is an isomorphism.  Therefore the composite map
\begin{equation*}
  \psi_A := \alpha_A \circ q^{-1}_{\cL_A} \circ \ev_{E_{\cC/\cL_A}^\oplus} :\cZ_\cC(A) \xrightarrow{\;\cong\;} \cF(A)
\end{equation*}
is an isomorphism of commutative rings by Lemma \ref{lem:GenCentre} and Thm.\ \ref{thm:BerCent}. Prop.\ \ref{prop:FtoZmap} now tells us that these maps $\psi_A$ commute with the restriction maps in the presheaves $\cZ_\cC$ and $\cF$ and therefore combine to give an isomorphism of presheaves $\psi : \cZ_\cC \xrightarrow{\;\cong\;} \cF$. However $\cF$ is a sheaf by Lemma \ref{lem:Fsheaf}, so $\cZ_\cC$ must be a sheaf as well.
\end{proof}

\section{Other cases}\label{sec:other}

There are two more cases where $\mathbf{Sp}(\cC)$ can be used for an alternative interpretation of the central sheaf $Z_\cC$.

\subsection{The locally finitely presented case}\label{subsec:loc-fp}

Throughout this section we make the weaker assumption that $\cC$ is locally finitely presented. This means that the Grothendieck category has a set of finitely presented generators. Then any object of $\cC$ is a filtered colimit of finitely presented objects.

Recall from Remark \ref{rem:loc-noeth} that in the locally noetherian case all section functors commute with filtered colimits. This is no longer the case for a general locally finitely presented $\cC$. Therefore one restricts attention to the localizing subcategories $\cL$ which are of finite type. By definition this means that the corresponding section functor commutes filtered colimits whose transition maps are monomorphisms.

\begin{remark}\label{rem:loc-fp}
   For any localizing subcategory $\cL$ of finite type in $\cC$, the quotient category $\cC/\cL$ is, in general, only locally finitely generated.
\end{remark}
\begin{proof}
\cite{Ga0} Thm.\ 5.8 (compare also Prop.\ 5.9).
\end{proof}

\begin{proposition}\phantomsection\label{prop:locfp-classific}
\begin{itemize}
  \item[i.] There is a topology on $\mathbf{Sp}(\cC)$ such that
\begin{align*}
  \text{collection of all localizing}\ \quad & \xrightarrow{\;\simeq\;} \text{set of all closed} \\
  \text{subcategories of finite type of $\cC$} &       \qquad             \text{subsets of $\mathbf{Sp}(\cC)$}  \\
                          \cL & \longmapsto A(\cL)       \nonumber
\end{align*}
is an inclusion reversing bijection.
  \item[ii.] For any localizing subcategories of finite type $\cL_1$, $\cL_2$, and $\{\cL_i\}_{i \in I}$ of $\cC$ we have:
\begin{itemize}
  \item[a)] $\cL_1 \cap \cL_2$ and $\bigvee_{i \in I} \cL_i$ are of finite type;
  \item[b)] $A(\cL_1 \cap \cL_2) = A(\cL_1) \cup A(\cL_2)$ and $A(\bigvee_{i \in I} \cL_i) = \bigcap_{i \in I} A(\cL_i)$.
\end{itemize}
\end{itemize}
\end{proposition}
\begin{proof}
For i. and ii.a) see \cite{Gar} Thm.\ 11 and its proof. ii.b) then follows immediately.
\end{proof}

\begin{corollary}\label{cor:fpft-stable}
 Finite unions and intersections of stable closed subsets of $\mathbf{Sp}(\cC)$ are stable.
\end{corollary}
\begin{proof}
The above Prop.\ \ref{prop:locfp-classific}.ii together with Prop.\ \ref{prop:vee-stable}.
\end{proof}

\begin{corollary}
The presheaf $A \mapsto Z(\cC/\cL_A)$ on $\mathbf{Sp}(\cC)$ has the sheaf property for finite coverings of stable closed subsets by stable closed subsets.
\end{corollary}
\begin{proof}
This follows now directly from Prop.\ \ref{prop:finite-sheaf}.
\end{proof}

That in the locally noetherian case we could prove the sheaf property for infinite coverings as well relied very much on the fact that in this case any injective object is a direct sum of indecomposables. This fails for any $\cC$ of a more general kind. But one of the basic reasons that Prop.\ \ref{prop:locfp-classific} works is that a locally finitely generated category is cogenerated by its indecomposable injective objects in the sense of \cite{Kra} Lemma 3.1.

We also mention that the more restrictive case where $\cC$ is locally coherent was treated in \cite{Her} and \cite{Kra}. In this situation the section functor of a localizing subcategory of finite type even commutes with filtered colimits and the corresponding quotient category again is locally coherent (\cite{Ga0} Thm.\ 5.14).

\subsubsection{A basic case}\label{ssubsec:ModG}

An important class of Grothendieck categories where Prop.\ \ref{prop:locfp-classific} and its corollaries apply is the following. Let $G$ be a locally profinite group and $k$ be a field. We let $\Mod(G)$ denote the category of smooth $G$-representations in $k$-vector spaces. (Recall that a $G$-representation $V$ is called smooth if the stabilizer of any vector in $V$ is open in $G$.) This obviously is an (AB5) abelian category.

\begin{lemma}
  Let $\mathfrak{U}$ be a fundamental system of compact open subgroups of $G$. Then the representations $k[G/U]$, for $U \in \mathfrak{U}$, form a set of generators of $\Mod(G)$.
\end{lemma}
\begin{proof}
Let $\alpha : V \rightarrow V'$ be a non-zero map and choose a $v \in V$ such that $\alpha(v) \neq 0$ as well as a $U \in \mathfrak{U}$ which fixes $v$. Then the map $\alpha^U : V^U \rightarrow V'^U$ on $U$-fixed vectors is non-zero. It remains to note that $\Hom_{\Mod(G)}(k[G/U],-) = (-)^U$.
\end{proof}

It follows that $\Mod(G)$ is a Grothendieck category. It clearly is locally of finite type.

\begin{lemma}
  Let $U \subseteq G$ be a compact open subgroup and $M$ be a finite dimensional representation in $\Mod(U)$. Then the compact induction $\ind_U^G(M)$ is finitely presented in $\Mod(G)$.
\end{lemma}
\begin{proof}
 We have
\begin{equation*}
  \Hom_{\Mod(G)}(\ind_U^G(M),V) = \Hom_{\Mod(U)}(M,V) = (\Hom_k(M,k) \otimes_k V)^U \ .
\end{equation*}
Now use \cite{Pop} \textbf{Thm.}\ 3.5.10 observing that the functor $(-)^U$ of $U$-invariants commutes with filtered colimits.
\end{proof}

\begin{proposition}
  The category $\Mod(G)$ is locally finitely presented.
\end{proposition}
\begin{proof}
  This is immediate from the above two lemmas.
\end{proof}

\begin{remark}
  In \cite{Sho} it is shown that for $G = SL_2(F)$ with $F/\mathbb{Q}_p$ any finite extension and $k$ a finite field of characteristic $p$ the category $\Mod(G)$ is locally coherent. Recently this has been extended to $G = GL_2(F)$ in \cite{Tim}.
\end{remark}

\subsection{The locally coirreducible cases}\label{subsec:loc-ci}

Throughout this section we assume that $\cC$ is locally coirreducible. For the definition we refer to \cite{Pop} p.\ 330. Any locally noetherian $\cC$ and, more generally, any $\cC$ which has a Krull dimension in the sense of \cite{Gab} p.\ 383 is locally coirreducible by \cite{Pop} Thm.\ 5.5.5. In contrast, a locally coherent category need not be locally coirreducible.

\begin{lemma}\label{rem:loc-coirr}
   For any localizing subcategory $\cL$ of $\cC$ the categories $\cL$ and $\cC/\cL$ are locally coirreducible as well.
\end{lemma}
\begin{proof}
\cite{Pop} Prop.\ 5.3.6.
\end{proof}

For the convenience of the reader we first recall the following elementary concept. Let $S$ be any set. A \textit{topological closure operator} on $S$ is a selfmap $A \mapsto \overline{A}$ on the power set $\mathfrak{P}(S)$ satisfying:
\begin{enumerate}
  \item $\overline{\emptyset} = \emptyset$,
  \item $A \subseteq \overline{A}$ (in particular, $\overline{X} = X$),
  \item $\overline{\overline{A}} = \overline{A}$,
  \item $\overline{A \cup B} = \overline{A} \cup \overline{B}$.
\end{enumerate}
Note that $(4)$ implies
\begin{enumerate}
  \item[(5)] $A \subseteq B \Longrightarrow \overline{A} \subseteq \overline{B}$.
\end{enumerate}
A subset $A \subset S$ is called \textit{closed} if $\overline{A} = A$.

Consider an arbitrary family $\{A_i\}_{i \in I}$ of closed subsets and put $A := \bigcap_{i \in I} A_i$. Then
\begin{equation*}
  A \subseteq \overline{A} \subseteq \bigcap_{i \in I} \overline{A_i} = \bigcap_{i \in I} A_i = A \ .
\end{equation*}
It follows that $A$ is closed. This shows that there is a unique topology on $S$ whose closed subsets are the closed subsets in the above sense.

For any localizing subcategory $\cL$ of $\cC$ we earlier defined the subset $A(\cL) \subseteq \mathbf{Sp}(\cC)$. Vice versa, for any subset $A \subseteq \mathbf{Sp}(\cC)$, we define
\begin{equation*}
  \cL_A := \ \text{localizing subcategory of all $Y$ such that $\Hom_\cC(Y,E) = 0$ for any $[E] \in A$}.
\end{equation*}

\begin{proposition}\phantomsection\label{prop:lc-closed}
\begin{itemize}
  \item[i.] $\cL = \cL_{A(\cL)}$ for any localizing subcategory $\cL$;
  \item[ii.] the map $A \mapsto A(\cL_A)$ is a topological closure operator on $\mathbf{Sp}(\cC)$.
\end{itemize}
\end{proposition}
\begin{proof}
i. We obviously have
\begin{equation*}
  \cL \subseteq \cL_{A(\cL)}  \qquad\text{and}\qquad   A(\cL) \subseteq A(\cL_{A(\cL)}) \subseteq A(\cL) \ ,
\end{equation*}
hence
\begin{equation}\label{f:AL}
  A(\cL_{A(\cL)}) = A(\cL) \ .
\end{equation}
It follows from \cite{Gab} Cor.\ 2 on p.\ 375 that the section functor $\cC/\cL \rightarrow \cC$ induces a bijection $\mathbf{Sp}(\cC/\cL) \xrightarrow{\simeq} A(\cL)$. (Note that the above even holds for a general Grothendieck category.) Therefore \cite{Pop} Cor.\ 5.3.8 says that \eqref{f:AL} implies that $\cL = \cL_{A(\cL)}$.

ii. \cite{Pop} Cor.\ 5.3.9.
\end{proof}

We call a subset of $\mathbf{Sp}(\cC)$ closed if it is closed w.r.t.\ the closure operator in the above Prop.\ \ref{prop:lc-closed}. Then part i. of this proposition says that all subsets of the form $A(\cL)$ are closed and part ii. says that any closed subset $A = \overline{A} = A(\cL_A)$ comes from a localizing subcategory. It follows that
\begin{align*}
  \text{collection of all localizing subcategories of $\cC$} & \xrightarrow{\;\simeq\;} \text{set of all closed subsets of $\mathbf{Sp}(\cC)$} \\
  \cL & \longmapsto A(\cL)
\end{align*}
is a bijection. As in section \ref{subsec:loc-fp} we deduce

\begin{corollary}
The presheaf $A \mapsto Z(\cC/\cL_A)$ on $\mathbf{Sp}(\cC)$ has the sheaf property for finite coverings of stable closed subsets by stable closed subsets.
\end{corollary}

\end{document}